\documentclass[12pt]{amsart}
\usepackage{amssymb,amsthm,enumerate, wasysym}
\usepackage{array,arydshln}
\usepackage{lscape}
\usepackage{pstricks, pst-plot, pst-node}
\usepackage{hyperref}                         
\usepackage[all]{xy}

\usepackage{tikz}                                 
\usetikzlibrary{angles}
\usetikzlibrary{decorations.markings,calc}
\usetikzlibrary{cd}

\usepackage[textwidth=1.8cm]{todonotes}
\reversemarginpar

\input xy
\xyoption{all}

\hypersetup{pdftex,                           
bookmarks=true,
pdffitwindow=true,
colorlinks=true,
citecolor=black,
filecolor=black,
linkcolor=black,
urlcolor=blue,
hypertexnames=true}

\newcommand{\bG}{{\mathbf{G}}}      

\newcommand{\cO}{{\mathcal{O}}}

\newcommand{\bB}{{\mathbf{B}}}

\newcommand{\bL}{{\mathbf{L}}}

\newcommand{\bH}{{\mathbf{H}}}
\newcommand{\bT}{{\mathbf{T}}}
\newcommand{\bS}{{\mathbf{S}}}
\newcommand{\bV}{{\mathbf{V}}}
\newcommand{\bZ}{{\mathbf{Z}}}
\newcommand{\bb}{{\mathbf{b}}}
\newcommand{\bc}{{\mathbf{c}}}

\newcommand{\dade}{{\mathbf{D}_k}}

\def\GL{ \text{\rm GL} }
\def\GU{ \text{\rm GU} }

\def\SO{ \text{\rm SO} }
\def\Sp{ \text{\rm Sp} }

\def\SL{ \text{\rm SL} }
\def\PSL{ \text{\rm PSL} }
\def\PGL{ \text{\rm PGL} }
\def\SU{ \text{\rm SU} }

\def\SO{ \text{\rm SO} }
\def\Spin{ \text{\rm Spin} }

\DeclareMathOperator{\Res}{Res}               
\DeclareMathOperator{\Inf}{Inf}                    

\DeclareMathOperator{\Irr}{Irr}
\DeclareMathOperator{\Gal}{Gal}

\DeclareMathOperator{\sgn}{sgn}

\newtheorem{thm}{Theorem}[subsection]
\newtheorem{lem}[thm]{Lemma}

\newtheorem{cor}[thm]{Corollary}

\theoremstyle{theorem}

\theoremstyle{definition}
\newtheorem{exmp}[thm]{Example}
\newtheorem{defn}[thm]{Definition}

\theoremstyle{remark}
\newtheorem{rem}[thm]{Remark}

\begin{document}


\title[On the source algebra equivalence class of blocks with cyclic defect groups, II]
{On the source algebra equivalence class of blocks with cyclic defect groups, II}

\date{\today}

\author{
Gerhard Hiss and  Caroline Lassueur
}
\address{{\sc Gerhard Hiss},  Lehrstuhl f\"ur Algebra und Zahlentheorie, 
RWTH Aachen University, 52056 Aachen, Germany.}
\email{gerhard.hiss@math.rwth-aachen.de}
\address{{\sc Caroline Lassueur}, 
RPTU Kaiserslautern--Land\-au, Fachbereich Mathematik,
67653 Kaiserslautern, Germany and Leibniz Universit\"at Hannover,
Institut f\"ur Algebra, Zahlentheorie und Diskrete Mathematik,
Welfengarten 1, 30167 Hannover, Germany.}
\email{lassueur@mathematik.uni-kl.de}


\begin{abstract} 
This series of papers is a contribution to the program of classifying
$p$-blocks of finite groups up to source algebra equivalence, starting
with the case of cyclic blocks.
To any $p$-block $\bB$ of a finite group with cyclic defect group $D$, 
Linckelmann associated an invariant $W( \bB )$, which is an indecomposable 
endo-permut\-ati\-on module over~$D$, and which, together with the Brauer tree
of~$\bB$, essentially determines its source algebra equivalence class.

In Parts II--IV of our series of papers, we classify, for odd~$p$, those 
endo-permutation modules of cyclic $p$-groups arising from $p$-blocks of 
quasisimple groups.

In the present Part II, we reduce the desired classification for the quasisimple 
classical groups of Lie type $B$, $C$, and $D$ to the corresponding
objective for the general linear and unitary groups; the classification is 
completed for the latter groups.
\end{abstract}

\thanks{The second author gratefully acknowledges financial support by the SFB TRR 195.}

\maketitle


\pagestyle{myheadings}
\markboth{On the source algebra equivalence class of blocks with cyclic defect groups, II}
{On the source algebra equivalence class of blocks with cyclic defect groups, II}

\vspace{6mm}
\section{Introduction}

Let~$k$ be an algebraically closed field of characteristic $p > 0$. By a 
\emph{$p$-block}, or simply a \emph{block}, we mean an indecomposable direct 
algebra factor of the group algebra~$kG$ for some finite group~$G$. A block is 
\emph{cyclic}, if it has a cyclic defect group.

Our work falls into the general program of classifying blocks up to certain 
categorical equivalences, which, in turn, is motivated by the following basic 
question: Can we determine which $k$-algebras occur as blocks? The approach to 
this question is guided by two concepts of equivalence, the general notion of 
\emph{Morita equivalence} for finite-dimensional $k$-algebras, and the stronger 
notion of \emph{source-algebra equivalence}, which applies to blocks only. In 
this respect, two prominent conjectures, one by Donovan and one by Puig, predict 
that the number of Morita equivalence classes, respectively source-algebra 
equivalence classes, of blocks is finite, provided a defect group is fixed. We 
notice, however, that classifying blocks up to Morita equivalence, respectively 
source-algebra equivalence is a much more involved task than proving the 
finiteness of the set of equivalence classes.

Donovan's conjecture is known to hold for a fairly good number of small defect 
groups. An excellent summary of recent works and advances is given by 
the fast-growing database of blocks maintained by Charles Eaton \cite{EaWiki}. 
On the side of source-algebra equivalences much less is known: Puig's conjecture 
is known to hold when the defect groups are cyclic for an arbitrary prime 
number~$p$ by~\cite{Li96}, and for Klein-four defect groups when $p = 2$ 
by~\cite{CEKL}. In the latter case,~\cite{CEKL} also provides us with a 
classification of blocks up to source-algebra equivalence. In the cyclic case, 
though it may sound like everything is known, no such classification is 
available at present. In fact, we do not even have a classification of cyclic 
blocks up to Morita equivalence.

The Morita equivalence class of a cyclic block~$\mathbf{B}$ is encoded by its 
embedded Brauer tree. The source-algebra equivalence class of $\mathbf{B}$ is 
encoded by its embedded Brauer tree together with a sign function on the 
vertices of this tree and an invariant $W(\mathbf{B})$ associated to~$\bB$ by 
Linckelmann; see~\cite{Li96}. This is a certain indecomposable endo-permutation 
$kD$-module, 
where~$D$ is a defect group of~$\bB$. In~\cite{HL24}, which constitutes Part~I 
of our series of papers, we investigated the modules~$W( \bB )$ and started, for 
odd~$p$, the classification of those~$W( \bB )$ that arise from a cyclic 
block~$\bB$ of a quasisimple group~$G$. We proved that~$W( \bB )$ is trivial, in 
the sense that $W( \bB ) \cong k$, in a large number of cases. For an accurate 
account see \cite[Section~$6$]{HL24}. In particular,~$W( \bB )$ can only be 
non-trivial if~$G$ is a finite group of Lie type of characteristic different 
from~$p$, and if $G/Z(G)$ does not have an exceptional Schur multiplier;
moreover, if~$G$ is an exceptional group of Lie type, then $p = 3$, or~$G$ is of 
type~$E_8$ and $p = 3$ or~$5$.

In~\cite{HL24} we have announced one subsequent article containing the  
classification of the non-trivial invariants~$W( \bB )$. It has turned out that 
it is more appropriate to further divide this material into three portions, 
making up Parts II -- IV of our series of papers. This division follows three 
well-defined, methodologically disjoint, sections. 

Let us now describe the content of the present Part~II in more detail. In 
Section~$2$ we provide a large collection of preliminary results, which will 
also be used in Parts~III and~IV. Our classification program starts in 
Section~$3$. In order to describe the results, suppose henceforth, until 
otherwise said, that~$G$ is a quasisimple group of Lie type of characteristic
different from~$p$, and that~$\bB$ is a cyclic $p$-block of~$G$ with defect 
group~$D$. We also choose a suitable algebraic group~$\bG$, such that~$G$ is the 
set of fixed points of some Steinberg morphism of~$\bG$. By the results 
summarized two paragraphs above, every group we are left to consider is of this 
form.

In \cite[Section~$3$]{HL24} we showed how to determine $W(\bB)$ from the
character table of~$G$, in particular from the sign sequence of a 
non-exceptional character of~$\bB$ on the powers of a generator of~$D$. This 
result is one of our main tools. Namely, suppose that such a non-exceptional 
character is an irreducible Deligne--Lusztig character, arising from a linear
character of $p'$-order of a maximal torus~$T$ of~$G$ containing~$D$. Then the 
character formula for Deligne--Lusztig characters allows for a computation of 
this sign sequence; see Lemma~\ref{SignAndOmegaInvariant}. The relative ranks of 
the $\bG$-centralizers of the elements of~$D$ play a crucial role here. This 
approach is particularly fruitful if the block~$\bB$ is regular in the sense of 
Brou{\'e}; see \cite[Th{\'e}o\-r{\`e}\-me~$3.1$]{brouMo}.

Let~$D_1$ denote the unique subgroup of~$D$ of order~$p$. By its very 
definition, $W( \bB ) \cong W( \bc )$, where~$\bc$ is a Brauer correspondent 
of~$\bB$ in $C_G( D_1 )$. This observation is our second main tool.

Suppose that~$G$ is a quasisimple classical group different from $\SL_n( q )$
or $\SU_n(q)$. Our main reduction, presented in Section~$3$, shows that there is 
a block ${\bB}'$ of a general linear or unitary group~${G}'$, such that $\bB$ 
and~${\bB}'$ have isomorphic defect groups, and that 
$W( \bB ) \cong W( {\bB}' )$ after an identification of the defect groups; see 
Theorem~\ref{MainCorReduction}. So even though we are interested in 
quasisimple groups, we now have to investigate the general linear and unitary 
groups. This is achieved in Section~$4$.

It turns out a posteriori, that the general linear groups do not yield examples 
of blocks~$\bB$ with $W( \bB ) \not\cong k$. Thus let $G = \GU_n( q )$ for
some prime power~$q$ with $p \nmid q$ and some integer $n \geq 2$. The case 
$p \mid q + 1$ is essential, in the sense that the general case can be reduced 
to this; see Theorem~\ref{GLNQdneq1}. Suppose that $p \mid q + 1$. Then 
$p \mid |Z(G)|$, so that $G = C_G( D_1 )$. The $p$-blocks of~$G$ have been 
determined by Fong and Srinivasan in~\cite{fs2}. In our case, we may assume 
that~$\bB$ is a regular block with respect to a cyclic torus~$T$ of~$G$ of order 
$q^n - (-1)^n$ containing a defect group of~$\bB$. Then the non-exceptional 
character of~$\bB$ is an irreducible Deligne--Lusztig character arising from a 
linear character of~$T$ of $p'$-order and in general position. As indicated 
above, this allows us to determine~$W( \bB )$. We find that $W(\bB) \cong k$, 
unless~$n$ is odd and $p \equiv -1\,\,(\mbox{\rm mod}\,\,4)$; see 
Corollary~\ref{GLNQdeq1Source}. In this case, the module~$W(\bB)$ is described 
in terms of the parameters~$p$,~$q$ and~$n$.

As a result of our investigations, we present an explicit example where the 
Bonnaf{\'e}--Dat--Rouquier Morita equivalence \cite[Theorem~$1.1$]{BoDaRo} does 
not preserve the source algebras; see Example~\ref{BoDaRoExample}. This ends 
the description of the results of Part~II.

In Part~III we will determine the non-trivial invariants~$W( \bB )$ arising from 
cyclic blocks of quasisimple groups~$G$ with $G/Z(G) \cong \SL_n(q)$ 
or~$\SU_n(q)$. Finally, Part~IV achieves the desired classification for the
exceptional groups of Lie type.

\section{Preliminaries}

Throughout this article, we use the notation introduced in 
\cite[Section~$2$]{HL24}. To simplify the formulation of our results, we
present some specific additional notation.

\subsection{Some general notation} Here, we set up some notation that will be 
used in the formulation of our results.

\begin{defn}
\label{SignFunction}
{\rm
Let~$X$ be a set and $\chi: X \rightarrow \mathbb{R}$ a map.
Put $\sigma_{\chi} := \sgn \circ \chi$, where $\sgn$ is the sign 
function on the real numbers, that is, for $x \in X$,
$$\sigma_\chi( x ) := \begin{cases} +1, & \text{\ if\ } \chi( x ) > 0, \\
                                     0, & \text{\ if\ } \chi( x ) = 0, \\
                                    -1, & \text{\ if\ } \chi( x ) < 0.
                           \end{cases}$$
}\hfill $\Box$
\end{defn}

\begin{defn}
\label{PowersOfT}
{
Let~$H$ be a finite group,~$p$ a prime and $t \in H$ a $p$-element. Let~$X$ be a 
set and let $\rho: H \rightarrow X$ be a map. Define, for every positive 
integer~$m$, an element $\rho^{[m]}( t ) \in X^m$ by
$$\rho^{[m]}( t ) := (\rho( t^{p^{m-1}} ), \rho( t^{p^{m-2}} ), \ldots,
\rho( t^p ), \rho( t ) ).$$
Thus, $\rho^{[m]}( t )$ contains, in reverse order, the values of~$\rho$
at the elements $t, t^p, t^{p^2}, \ldots, t^{p^{m-1}}$.
}\hfill $\Box$
\end{defn}

In the notation of Definition~\ref{PowersOfT}, suppose that $|t| = p^l$. Then 
$|t^{p^{m-j}}| = p^{l - \min\{ l, m - j \}}$ for
$1 \leq j \leq m$. If $m \leq l$, then $|t^{p^{m-j}}| = p^{l-m+j}$ for 
$1 \leq j \leq m$; if $m > l$, then $|t^{p^{m-j}}| = 1$ for 
$1 \leq j \leq m - l$, and $|t^{p^{m-j}}| = p^{j - (m - l)}$ for 
$m -l < j \leq m$.

\subsection{Labels}
\label{Labels}
In the following, we let $\mathbb{F}_2 = \{ 0, 1 \}$ denote the field with~$2$
elements, let~$l$ be a positive integer, and put $\Lambda := \{ 0, 1, \ldots , l - 1 \}$. 
By $\mathcal{P}(\Lambda)$ we denote the power set of~$\Lambda$. The symmetric 
difference of sets equips~$\mathcal{P}(\Lambda)$ with the structure of an 
abelian group of exponent two, thus of an $\mathbb{F}_2$-vector space of 
dimension~$l$.

We write the elements of $\mathbb{F}_2^{\Lambda}$ as tuples 
$(\alpha_0, \alpha_1, \ldots, \alpha_{l-1} )$. For 
$A \subseteq \Lambda$ we let $\mathbf{1}_A \in \mathbb{F}_2^\Lambda$ denote the 
characteristic function of~$A$, i.e.\ $\mathbf{1}_A = 
(\alpha_0, \alpha_1, \ldots, \alpha_{l-1} )$ with $\alpha_j = 1$ if and only if 
$j \in A$. Then the map
\begin{equation}
\label{CharacteristicFunction}
\mathcal{P}(\Lambda) \rightarrow \mathbb{F}_2^{\Lambda},\quad\quad A \mapsto \mathbf{1}_A
\end{equation}
is an isomorphism of vector spaces.

The set $\{ -1, 1 \}^l \subseteq \mathbb{Z}^l$ is an $\mathbb{F}_2$-vector space
with component-wise multiplication. We define 
\begin{equation}
\label{DefineOmega}
\omega_{\Lambda}: \mathbb{F}_2^{\Lambda} \rightarrow \{ -1, 1 \}^l
\end{equation}
by 
$$\omega_{\Lambda}(\alpha_0, \ldots , \alpha_{l-1})_i = 
\begin{cases} 
+1, \text{\ if\ } \sum_{j = 0}^{i-1} \alpha_j = 0\\
-1, \text{\ if\ } \sum_{j = 0}^{i-1} \alpha_j = 1
\end{cases}\quad\quad\text{for\ } 1 \leq i \leq l.
$$
Then $\omega_{\Lambda}$ is an $\mathbb{F}_2$-vector space isomorphism.

A set $I \subseteq \Lambda$ is called an \textit{interval}, if~$I$ is the 
intersection of an interval of real numbers with~$\Lambda$. A non-empty 
interval~$I$ is written as $I = [a,b]$, if~$a$, respectively~$b$ is the smallest, 
respectively the largest element of~$I$. If $a, b \in \Lambda$ with $a > b$, 
then $[a,b]$ denotes the empty interval. The \textit{distance} of two
intervals is the Euclidean distance of subsets of the reals.

\begin{lem}
\label{OmegaOfInterval}
Let $I = [a,b] \subseteq \Lambda$ be an interval with $a \leq b$.
Then
$$
\omega_{\Lambda}( \mathbf{1}_I )_i = 
\begin{cases}
1, & 1 \leq i \leq a, \\
(-1)^{i - a}, & a + 1 \leq i \leq b, \\
(-1)^{b - a + 1}, & b + 1 \leq i \leq l,
\end{cases}
$$
and $\omega_{\Lambda}( \mathbf{1}_{\emptyset} )_i = 1$ for all 
$1 \leq i \leq l$.
\end{lem}
\begin{proof}
This is clear from the definition.
\end{proof}

\subsection{The Dade group and the invariant $W( \bB )$}
Let~$p$ be an odd prime and let~$k$ denote an algebraically closed field of
characteristic~$p$. Let~$D$ denote a cyclic group of order~$p^l > 1$. Recall 
that for $1 \leq j \leq l$, the unique subgroup of~$D$ of order~$p^j$ is denoted 
by~$D_j$. Let $\Lambda = \{ 0, 1, \ldots , l - 1 \}$ as in 
Subsection~\ref{Labels}. Recall from \cite[Subsection~$3.1$]{HL24} that the Dade 
group $\dade(D)$ of~$D$ is isomorphic to $\mathbb{F}_2^l$ and consists of the 
isomorphism classes of the indecomposable capped endo-permutation $kD$-modules 
$$W_D( \alpha_0, \ldots , \alpha_{l-1}) := 
\Omega_{D/D_0}^{\alpha_0}\circ\Omega_{D/D_1}^{\alpha_1}\circ\cdots
\circ\Omega_{D/D_{l-1}}^{\alpha_{l-1}}(k),$$ 
with $\alpha_j \in \mathbb{F}_2$ for $j \in \Lambda$, where $\Omega_{D/D_j}$
denotes the relative syzygy operator with respect to $D_j \leq D$. 
Moreover, the addition in $\dade(D)$ is given by
$$W_D( \alpha_0, \ldots , \alpha_{l-1}) + W_D( \alpha_0', \ldots , \alpha_{l-1}') 
= W_D( \alpha_0 + \alpha_0', \ldots , \alpha_{l-1} + \alpha_{l-1}'),$$
so that the map $\mathbb{F}_2^{\Lambda} \rightarrow \dade(D), 
(\alpha_0, \ldots , \alpha_{l-1}) \mapsto W_D( \alpha_0, \ldots , \alpha_{l-1})$ 
is an isomorphism.

For later purposes, it is convenient to introduce a slightly different notation 
for these modules.  We use the bijection~(\ref{CharacteristicFunction}) to 
replace the label $(\alpha_0, \ldots , \alpha_{l-1}) \in \mathbb{F}_2^{\Lambda}$ 
of an element of $\dade(D)$ by a subset of~$\Lambda$. Let us give a formal 
definition.

\begin{defn}
\label{SpecialWs}
{\rm
For $A \subseteq \Lambda$ we put $W_D(A) := W_D(\mathbf{1}_A)$.
In particular, $W_D( \emptyset ) \cong k$.
}\hfill $\Box$
\end{defn}

\begin{lem}
\label{NewNotation}
Let $A \subseteq \Lambda$ and put $W := W_D(A)$. Let~$t$ denote a
generator of~$D$. Then, in the notation of {\rm Definition~\ref{PowersOfT}},
$$\omega_W^{[l]}(t) = \omega_{\Lambda}( \mathbf{1}_A ),$$
with $\omega_W := \sigma_{\rho_W}$, where~$\rho_W$ is the ordinary character of
the lift of determinant~$1$ of~$W$; see \cite[Subsection~$3.1$]{HL24}.
In particular,~$W$ is uniquely determined, up to isomorphism, 
by~$\omega_W^{[l]}(t)$.
\end{lem}
\begin{proof}
The first statement follows from \cite[Lemme~$3.2$]{HL24} 
and~(\ref{DefineOmega}), the second from the 
fact that~$\omega_{\Lambda}$ is an isomorphism.
\end{proof}

Let~$G$ be a finite group and let~$\bB$ be a $p$-block of~$G$ with a non-trivial
cyclic defect group~$D$. We then write $W( \bB )$ for the indecomposable capped 
endo-permutation $kD$-module associated to~$\bB$. Recall that $W( \bB )$ is the 
common source of the simple $k N_G( D_1 )$-modules, or equivalently, the simple 
$kC_G(D_1)$-modules in the Brauer correspondents of~$\bB$ in $N_G( D_1 )$, 
respectively $C_G( D_1)$; see \cite[Subsection~$2.4$]{HL24}. In 
particular,~$D_1$ acts trivially on $W( \bB )$.  The isomorphism class 
of~$W( \bB )$ is an element of~$\dade(D)$. 

Recall that~$\bB$ is nilpotent, if and only if it has a unique 
irreducible Brauer character. This is the case if $D_1 \leq Z(G)$.

We use the term exceptional character of~$\Irr( \bB )$ as implied by Dade's 
theorem on blocks with cyclic defect groups, following the account given in 
\cite[Theorem~$68.1$]{DornhoffB}. Namely, if $\Lambda$ is the set introduced
in \cite[Theorem~$68.1$(2)]{DornhoffB} (not the same $\Lambda$ as above),
and if $|\Lambda| \geq 2$, then the characters $\chi_{\lambda} \in \Irr( \bB )$, 
$\lambda \in \Lambda$, labelled as in \cite[Theorem~$68.1$(5)]{DornhoffB}, are 
the exceptional characters of~$\Irr( \bB )$. The other elements of $\Irr( \bB )$
are the non-exceptional characters. If $|\Lambda| \geq 2$ and there are more than 
one non-exceptional characters in $\Irr( \bB )$, the exceptional characters can 
be determined by the fact that they have the same restriction to the $p$-regular
classes of~$G$. The case of~$|\Lambda| \geq 2$ and a unique non-exceptional
character in $\Irr( \bB )$, is treated in 
Lemma~\ref{IdentifyingTheNonExceptional} below.

\begin{rem}
\label{SignSequenceRemark}
{\rm
With the notation introduced above, let $|D| = p^l$. Suppose that 
$D_1 \leq Z(G)$ so that $\Irr(\bB)$ contains a unique non-exceptional character, 
$\chi$, say. Put $W := W( \bB )$ and
let $t \in D$ be a generator. Then, \cite[Lemma~$3.3$]{HL24}, reformulated in 
the notation of Definition~\ref{PowersOfT}, implies that
$$\sigma_{\chi}^{[l]}( t ) = \omega_W^{[l]}( t ).$$
In particular,~$W$ is uniquely determined by $\sigma_{\chi}^{[l]}( t )$ up to
isomorphism, and if 
$\sigma_{\chi}^{[l]}( t ) = \omega_{\Lambda}( \mathbf{1}_A )$ for some 
$A \subseteq \Lambda \setminus \{ 0 \}$, then $W = W_D( A )$.
}\hfill{$\Box$}
\end{rem}
In order to apply this remark, it is necessary to identify the non-exceptional
character of a cyclic block under the given hypothesis. This can be achieved by 
looking at character values. Let $\mathcal{G}$ denote the subgroup of
$\Gal( \mathbb{Q}( \sqrt[|G|]{1} )/ \mathbb{Q} )$ which fixes the roots of unity
of $p'$-order. Then $\mathcal{G}$ acts on the set of characters of~$G$. A
character of $G$ is $p$-rational, if and only if it is fixed by $\mathcal{G}$.
\begin{lem}
\label{IdentifyingTheNonExceptional}
Assume $p$ is odd and  $\bB$ is a nilpotent block with a non-trivial cyclic defect group. Then, the unique non-exceptional
character of~$\Irr( \bB )$ is also the unique $p$-rational character of~$\Irr( \bB )$.
\end{lem}
\begin{proof}
The non-exceptional character in~$\Irr( \bB )$ is $p$-rational by
\cite[Theorem~68.1(8)]{DornhoffB}. The same reference implies that the matrix of 
generalized decomposition numbers of~$\bB$, as defined in 
\cite[Section IV.6, P.~$175$]{Feit}, has a unique $p$-rational column. Then 
\cite[Lemma IV.6.10]{Feit} shows that $\Irr(\bB)$ has a unique $p$-rational 
character. 
\end{proof}

\subsection{Preliminaries from representation theory}
\label{SpecialConfigurations}
We continue with the general notation of the first part~\cite{HL24} of this 
work. Thus~$G$ denotes a finite group,~$p$ is a prime and $(K,\cO,k)$ is a 
splitting $p$-modular system for~$G$, as introduced in 
\cite[Subsection~$2.1$]{HL24}. Except in Lemma~\ref{InflationGreen}
below, we assume that $p$ is odd throughout this subsection.

We specify our usage of the term Brauer correspondent. Suppose~$\bB$ is a 
$p$-block of~$G$ with an abelian defect group~$D$. If $D' \leq D$, any 
block~$\bb'$ of $C_G(D')$ such that $(D',\bb') \leq (D,\bb_0)$ for some 
maximal $\bB$-Brauer pair $(D,\bb_0)$, will be called \textit{a Brauer 
correspondent} of~$\bB$ in~$C_G(D')$. 

We continue with a result on the compatibility of Brauer correspondence and 
domination of blocks. For the concept of domination of blocks see 
\cite[Subsection~$5.8.2$]{NaTs}.

\begin{lem}
\label{InflationGreen}
Let $Y \unlhd Z(G)$ and set $\bar{G} := G/Y$. 
Let~$\bar{\bB}$ be a block of $k\bar{G}$ and let~$\bB$ be the unique 
block of~$kG$  dominating $\bar{\bB}$. Let~$D$ be a defect group of~$\bB$.
Then $\bar{D} := DY/Y$ is a defect group of~$\bar{\bB}$.

Let $H \leq G$ with $N_G( D ) \leq H$, and put $\bar{H} := H/Y$. Then
$N_{\bar{G}}(\bar{D}) \leq \bar{H}$.
Let~$\bb$ be the Brauer correspondent of~$\bB$ in~$H$ and let~$\bar{\bb}$ 
be the Brauer correspondent of~$\bar{\bB}$ in~$\bar{H}$. 
If~$D$ is abelian, then~$\bb$ is the unique block of~$kH$ 
dominating~$\bar{\bb}$.
\end{lem}
\begin{proof}
Since $G/Y \cong (G/O_{p'}(Y)/(Y/O_{p'}(Y))$, the group $\bar{D}$ is a defect 
group of~$\bar{\bB}$ by \cite[Theorems~$5.8.8$ and~$5.8.10$]{NaTs}.

There is a unique block~$\bb_{0}$ of~$kH$ 
dominating~$\bar{\bb}$. To prove that $\bb_{0} = \bb$, it suffices to show
that~$\bb$ contains a module which is the inflation from 
$\bar{H}$ to $H$ of an indecomposable 
$\bar{\bb}$-module.
\par
As~$\bB$ dominates~$\bar{\bB}$, there is a simple~$\bar{\bB}$-module~$\bar{V}$, 
whose inflation $V := \Inf_{\bar{G}}^{G}(\bar{V})$ to~$G$ belongs to~$\bB$.
The defect group~$D$ of~$\bB$ being abelian, it must be a vertex of~$V$ by 
Kn\"{o}rr's theorem (see~\cite[Corollary~$3.7$(ii)]{Knoerr79}), and 
similarly~$\bar{D}$ is a vertex of~$\bar{V}$. Write~$f(V)$, 
respectively~$f(\bar{V})$, for the Green correspondent of~$V$ in~$H$, 
respectively of~$\bar{V}$ in~$\bar{H}$. As the Green correspondence 
commutes with the Brauer correspondence,~$f(V)$ belongs to~$\bb$ 
and~$f(\bar{V})$ belongs to~$\bar{\bb}$. Finally, 
\cite[Proposition~$5$]{Harris08} yields
\[
f(V) \cong \Inf_{\bar{H}}^{H}(f(\bar{V})),
\]
as $Y \leq N_G( D ) \leq H$.
This proves our claim.
\end{proof}

In the following lemma we investigate the behavior of~$W( \bB )$ with respect to 
domination of blocks. Recall that if~$D$ is a non-trivial 
cyclic~$p$-group,~$D_1$ denotes its unique subgroup of order~$p$.

\begin{lem}
\label{lem:OpZG}
Let $Y \leq Z(G)$ with $O_p(Y) \neq \{ 1 \}$. Set $\bar{G} := G/Y$ and let~$a$ be
the positive integer such that $|O_{p}(Y)| = p^{a}$. 

Let~$\bar{\bB}$ be a block of~$\bar{G}$ and let~$\bB$ be the unique block of~$G$ 
dominating~$\bar{\bB}$. Suppose that a defect group~$D$ of~$\bB$ is cyclic of
order~$p^l$ with $l \geq 1$. 

Then the following assertions hold.

{\rm (a)} We have $G = N_{G}(D_{1}) = C_{G}(D_{1})$, the blocks~$\bB$
and~$\bar{\bB}$ are nilpotent, and $\bar{D} := D/O_{p}(Y)$  is a defect
group of $\bar{\bB}$.

{\rm(b)} We have $W(\bB) \cong \Inf_{\bar{D}}^{D}(W(\bar{\bB}))$ or
$W(\bB) \cong \Inf_{\bar{D}}^{D}(\Omega(W(\bar{\bB})))$.

{\rm (c)} Suppose that $W(\bB) = W_D(\alpha_0,\ldots,\alpha_{l-1}) = 
W_{D}( {\mathbf{1}}_{A})$ for some $A \subseteq \{ 0, \ldots , l - 1 \}$. 
Then
$\alpha_{0} = \cdots = \alpha_{a-1} = 0$, i.e.\ $A \subseteq \{a, \ldots, l-1\}$.
 
Setting $\bar{A} = \{ j - a \mid j \in A \setminus \{ a \} \} \subseteq 
\{ 1, \ldots , l - a - 1 \}$, we have 
$W( \bar{\bB} ) = W_{\bar{D}}( \bar{A} )$.
\end{lem}
\begin{proof}
(a) As~$D$ is a defect group and~$O_{p}(Y)$ is a normal $p$-subgroup
of~$Z(G)$, certainly~$O_{p}(Y)$ is contained in~$O_{p}(G) \leq D$. Thus
$D_{1} \leq O_{p}(Y) \leq D$ and so~$G$ centralizes~$D_{1}$. Thus~$\bB$ is
nilpotent and so must be~$\bar{\bB}$, proving the first assertion of~(a).
The second one follows from Lemma~\ref{InflationGreen}.

(b) By (a), up to isomorphism, there is a unique simple $\bB$-module, say~$V$,
on which~$D_{1}$ acts trivially. Thus, by definition,~$W(\bB)$ is a $kD$-source
of~$V$. Similarly, there is, up to isomorphism, a unique simple
$\bar{\bB}$-module, say~$\bar{V}$, and $\Inf_{\bar{G}}^G(\bar{V}) = V$ 
since~$\bB$ dominates~$\bar{\bB}$. In the quotient, we have two possibilities, 
namely either~$W(\bar{\bB})$ is a source of~$\bar{V}$ or~$\Omega(W(\bar{\bB}))$
is a source of~$\bar{V}$, depending on whether the elements of order~$p$ 
of~$\bar{D}$ act trivially on~$\bar{V}$, or not. This follows from the facts 
summarized in 
\cite[Subsections~$3.5$ and~$4.2$]{HL20}. However, if~$S$ denotes a 
$k\bar{D}$-source of~$\bar{V}$, then by \cite[Proposition~2]{Harris08} we have
$$W(\bB) \cong \Inf_{\bar{D}}^{D}(S).$$
Assertion (b) follows. 

(c) Suppose that 
$$W_{\bar{D}}( \bar{\alpha}_0, \ldots , \bar{\alpha}_{l-a-1} ) = 
\Omega_{\bar{D}/\bar{D}_0}^{\bar{\alpha}_0} \circ
\Omega_{\bar{D}/\bar{D}_1}^{\bar{\alpha_1}} \circ \cdots \circ
\Omega_{\bar{D}/\bar{D}_{l-a-1}}^{\bar{\alpha}_{l - a - 1}}(k)$$
represents the class of~$S$ in the Dade group of~$\bar{D}$. Then 
$W( \bB ) = \Inf_{\bar{D}}^{D}(S) = \Inf_{{D/D_a}}^{D}(S)$ is represented by
\[
W_D(\alpha_0,\ldots,\alpha_{l-1})
=\Omega_{D/D_0}^{\alpha_0}\circ\Omega_{D/D_1}^{\alpha_1}\circ\cdots
\circ\Omega_{D/D_{l-1}}^{\alpha_{l-1}}(k)
\]
with $\alpha_{0} = \cdots = \alpha_{a-1} = 0$ and $\alpha_{j + a} =
\bar{\alpha}_j$ for $0 \leq j \leq l - a - 1$.
Notice that $S = W( \bar{\bB} )$, if and only if $\bar{\alpha}_0 = 0$.
Otherwise, $\bar{\alpha}_0 = 1$ and $W( \bar{\bB} ) = \Omega( S )$.
Now $\Omega( S )$ is represented by
$W_{\bar{D}}( \bar{\alpha}_0', \ldots , \bar{\alpha}_{l-a-1}' )$ with
$\bar{\alpha}_0' = \bar{\alpha}_0 + 1$, and 
$\bar{\alpha}_j' = \bar{\alpha}_j$ for all $1 \leq j \leq l - a - 1$.
This proves our assertions.
\end{proof}

\subsection{Preliminaries on finite reductive groups}

Let~$\mathbb{F}$ denote an algebraic closure of a finite field of 
characteristic~$r$. The prime~$p$ introduced in 
Subsection~\ref{SpecialConfigurations} is assumed to be distinct from~$r$.
Let $\mathbf{G}$ be a connected reductive linear algebraic group 
over~$\mathbb{F}$ and 
let~$F$ denote a Steinberg morphism of~$\mathbf{G}$. Recall that if~$\mathbf{H}$ 
is a closed, $F$-stable subgroup of~$\mathbf{G}$, we write $H := \mathbf{H}^F$ 
for the set of $F$-fixed points of~$\mathbf{H}$, and $\mathbf{H}^\circ$ for the 
connected component of~$\mathbf{H}$. In particular, the group~$G$ introduced in
Subsection~\ref{SpecialConfigurations} now is of the form $G = \bG^F$. To avoid 
cumbersome double superscripts, we write 
$Z^{\circ}( \bG ) := Z( \bG )^{\circ}$ and $Z^{\circ}( G ) := 
Z^{\circ}( \bG )^F$. Also, if $s \in G$, we write $C^{\circ}_{\bG}( s ) :=
C_{\bG}( s )^{\circ}$ and $C^{\circ}_{G}( s ) := C^{\circ}_{\bG}( s )^F$.
Let~$\bG^*$ be a reductive group dual to~$\bG$, equipped with a dual Steinberg
morphism, also denoted by~$F$. If~$\bT$ and~$\bT^*$ are $F$-stable maximal tori 
of~$\bG$ and~$\bG^*$, respectively, and if $s \in T^*$ and $\theta \in \Irr(T)$
are such that the $G^*$-conjugacy class of $(\bT^*,s)$ and the $G$-conjugacy 
class of~$(\bT,\theta)$ correspond under the bijection exhibited in 
\cite[Proposition~$11.1.16$]{DiMi2}, we say that $(\bT^*, s)$ and 
$(\bT,\theta)$ \textit{are in duality}. When we talk about duality, we always 
tacitly assume that the necessary choices have been made. If $s \in G^*$ is 
semisimple, we let $\mathcal{E}( G, s )$ denote the corresponding Lusztig series 
of characters; see \cite[Definition~$2.6.1$]{GeMa}. 
The elements of $\mathcal{E}(G,1)$ are the unipotent characters of~$G$.
By a regular subgroup of~$\bG$ we mean an $F$-stable Levi subgroup of~$\bG$.

\subsubsection{Component groups}
\label{ComponentGroups}
Let $s \in \bG$ be a semisimple element. The finite group 
$$A_{\bG}( s ) := C_{\bG}(s)/C^{\circ}_{\bG}( s )$$
is called the \textit{component group} of~$s$. If~$s \in G$, then $A_{\bG}( s )$ 
is $F$-stable, and we write $A_{\bG}( s )^F$ for the set of its $F$-fixed 
points.  We have $A_{\bG}( s )^F = C_G( s )/C^{\circ}_G( s )$ by the 
Lang--Steinberg theorem. 

Let $s \in G^*$ be semisimple. There is an $F$-equivariant embedding 
of~$A_{\bG^*}( s )$ into $\Irr( Z( \bG )/Z^{\circ}( \bG ) )$; see \cite[$(8.4)$ 
and Lemme~$4.10$]{CeBo2}. In particular, $A_{\bG^*}( s )$ is abelian, and 
$A_{\bG^*}( s ) = 1$ if $Z( \bG )$ is connected. Moreover, every prime divisor 
of~$|A_{\bG^*}( s )|$ divides~$|s|$; see \cite[Lem\-ma~$8.3$]{CeBo2}.

\subsubsection{Regular embeddings and conjugacy classes}
Consider a regular embedding $\bG \rightarrow \tilde{\bG}$; see
\cite[Definition~$1.7.1$]{GeMa}. Then there is a bijection 
$\tilde{\bL} \mapsto \tilde{\bL} \cap \bG$ between the regular subgroups 
of~$\tilde{\bG}$ and those of~$\bG$. The inverse image of a regular 
subgroup~$\bL$ of~$\bG$ under this map equals $\bL Z( \tilde{\bG} )$.

\addtocounter{thm}{2}
\begin{lem}
\label{RegularEmbeddingsAndConjugacyClasses}
Let $\bG \rightarrow \tilde{\bG}$ be a regular embedding.

{\rm (a)} Let~$\tilde{\bL}$ be a regular subgroup of~$\tilde{\bG}$.
Then 
$$[\tilde{L}\colon\!L] = [\tilde{G}\colon\!G].$$

{\rm (b)} Let $s \in G$ be semisimple. Then $C_{\tilde{\bG}}( s )$ is a regular
subgroup of~$\tilde{\bG}$ if and only if~$C_{\bG}( s )$ is a regular subgroup 
of~$\bG$. If this condition is satisfied, the $G$-conjugacy class of~$s$ is a
$\tilde{G}$-conjugacy class.
\end{lem}
\begin{proof}
(a) Since~$\tilde{\bL}$ is a regular subgroup of~$\tilde{\bG}$, we have
$Z^{\circ}( \bG ) \leq \tilde{\bL}$. Moreover, 
$\tilde{\bL} = \bL Z( \tilde{\bG} )$ and $\tilde{\bG} = \bG Z( \tilde{\bG} )$. 
The claim follows from~\cite[Lemma~$1.7.8$]{GeMa}. 

(b) The first assertion is clear. The second follows from~(a), which implies
$$[\tilde{G}\colon\!C_{\tilde{G}}( s )] = [G\colon\!C_{G}( s )].$$
\end{proof}

\addtocounter{subsubsection}{1}
\subsubsection{A crucial invariant}
We now introduce a crucial invariant of a semisimple element of~$G$, based on 
the notion of relative rank.
The \textit{relative $F$-rank} of~$\mathbf{G}$, denoted by~$r_F( \mathbf{G} )$,
is defined as in~\cite[Definition~$7.1.5$]{DiMi2}. It can be computed from the 
order polynomial of the generic finite reductive group associated with the pair 
$(\mathbf{G}, F )$; see the discussion 
following~\cite[Definition~$2.2.11$]{GeMa}. 

\addtocounter{thm}{1}
\begin{exmp}
\label{OmegaOfGLNQ}
{\rm
Let~$n$ be a positive integer and~$q$ a prime power.

(a) If $(\mathbf{G}, F)$ is such that $G \cong \GL_n( q )$, 
then $r_F( \mathbf{G} ) = n$.

(b) If $(\mathbf{G}, F)$ is such that $G \cong \GU_n( q )$, 
then $r_F( \mathbf{G} ) = \lfloor n/2 \rfloor$.
}\hfill{$\Box$}
\end{exmp}

As usual, we also write $\varepsilon_{\mathbf{G}} := (-1)^{r_F( \mathbf{G} )}$.
Let us now define the aforementioned invariant.

\begin{defn}
\label{OmegaInvariant}
{\rm
For a semisimple element $s \in G$ put
$$\omega_{\bG}( s ) := 
\varepsilon_{\mathbf{G}}\varepsilon_{C^{\circ}_{\mathbf{G}}(s)}.$$
}\hfill{$\Box$}
\end{defn}

Let~$\bT$ denote an $F$-stable maximal torus of~$\bG$.
By $R_{\mathbf{T}}^{\mathbf{G}}$ we denote Deligne--Lusztig induction, which
maps generalized characters of~$T$ to generalized characters of~$G$; see
\cite[Subsection~$7.2$]{C2}. Also, recall Definition~\ref{SignFunction}.
\begin{lem}
\label{SignAndOmegaInvariant}
Let~$\bT$ denote an $F$-stable maximal torus of~$\bG$.
Let $\theta \in \Irr(T)$ be of $p'$-order and put
$\chi := 
\varepsilon_{\mathbf{G}}\varepsilon_{\mathbf{T}}R_{\mathbf{T}}^{\mathbf{G}}( \theta )$.
Let $t \in T$ be a $p$-element. Then $\chi(t)$ is a non-zero integer and
$\sigma_{\chi}(t) = \omega_{\bG}( t ) \in \{ -1, 1 \}$.
\end{lem}
\begin{proof}
Our hypothesis implies that $\theta( t' ) = 1$ for every $p$-element $t' \in T$.
By the character formula given in \cite[Proposition~$7.5.3$]{C2}, the value
of~$\chi(t)$ equals $\varepsilon_{\mathbf{G}}\varepsilon_{\mathbf{T}}
\varepsilon_{\mathbf{T}}\varepsilon_{C^{\circ}_{\mathbf{G}}(t)}$
times a positive rational number. Hence~$\chi(t)$ is an integer, since~$\chi$ is 
a generalized character. This gives our claims.
\end{proof}

We will also need a more specific version of Lemma~\ref{SignAndOmegaInvariant}.

\begin{lem}
\label{ValueOfDLCharacter}
Let~$\bT$ denote an $F$-stable maximal torus of~$\bG$ and let 
$\theta \in \Irr(T)$. Let $t \in T$ such that $\langle t \rangle$ is a Sylow 
$p$-subgroup of~$T$. Put $\bH := C_{\bG}( t )$. Then
$$R_{\mathbf{T}}^{\mathbf{G}}( \theta )(t') = 
\frac{\varepsilon_{\bT}\varepsilon_{\bH^{\circ}}|N_{G}( \bH )|}{|T||{\bH^{\circ}}^F|_{r}}$$
for every $t' \in \langle t \rangle$ with $\theta( t' ) = 1$ and
$C_{\bG}( t' ) = \bH$
(recall that~$r$ is the characteristic of~$\mathbb{F}$).
\end{lem}
\begin{proof}
Let~$t'$ be as in the assertion.
Let $g \in G$. We claim that $g^{-1}t'g \in T$, if and only if
$g \in N_{G}( \bH )$. Suppose first that $g \in N_{G}( \bH )$. Then
$g^{-1}t'g \in Z( \bH )$. Since $\bT$ is a maximal torus of~$\bH$, we have
$Z( \bH ) \leq \bT$, and thus $g^{-1}t'g \in T$. Suppose now that
$t'' := g^{-1}t'g \in T$. By hypothesis, 
$t'' \in \langle t \rangle \leq Z( \bH )$.
Hence $\bH \leq C_{\bG}( t'' ) = g^{-1}\bH g$.
It follows that $\bH = g^{-1}\bH g$, as~$g$ has finite order. 
This proves our claim.
Our assertion now follows from \cite[Proposition~$7.5.3$]{C2}.
\end{proof}

\addtocounter{subsubsection}{4}
\subsubsection{Isogenies}
An isogeny between algebraic groups is a surjective homomorphism with finite 
kernel. We record some properties which are transferred by isogenies.
\addtocounter{thm}{1}
\begin{lem}
\label{RelativeRankAndIsogenies}
Let~$\bG'$ be a connected reductive algebraic group over~$\mathbb{F}$ and 
let~$F'$ be a Steinberg morphism of~$\bG'$. 
Let $\nu : \bG \rightarrow \bG'$ be an isogeny with 
$\nu \circ F = F' \circ \nu$. Then the following statements hold.

{\rm (a)} Let $\bH \leq \bG$ be a closed subgroup. Then the restriction
          $\nu_{\bH} : \bH \rightarrow \nu( \bH )$ is an isogeny.

{\rm (b)} Let~$e$ be a positive integer and let~$\bS$ denote a $\Phi_e$-torus 
of~$\bG$. Then $\bS' := \nu( \bS )$ is a $\Phi_e$-torus of~$\bG'$. 
If~$\bL$ is an $e$-split Levi subgroup of~$\bG$, then $\bL' := \nu( \bL )$ is
an $e$-split Levi subgroup of~$\bG'$. (For the notions of $\Phi_e$-tori and 
$e$-split Levi subgroups see \cite{BrouMa}).

{\rm (c)} Let~$\bL'$ be an $F'$-stable Levi subgroup of~$\bG'$. Then there is 
an $F$-stable Levi subgroup~$\bL$ of~$\bG$ with $\nu(\bL ) = \bL'$.

{\rm (d)} We have $r_F( \bG ) = r_{F'}( \bG' )$.

{\rm (e)} We have $\nu( Z^{\circ}( \bG ) ) = Z^{\circ}( \bG' )$ and 
$|Z^{\circ}( \bG' )^{F'}|$ divides $|Z(G)|$. 

{\rm (f)} The map $\chi' \mapsto \chi' \circ \nu|_G$ is a bijection between the 
unipotent characters of $G' = {\bG'}^{F'}$ and those of $G = {\bG}^{F}$, which 
preserves the sets of $e$-cuspidal characters (see
 \cite[Definition~$3.5.19$]{GeMa}) for every positive integer~$e$.
\end{lem}
\begin{proof}
(a) This follows from the definition of an isogeny.

(b) Clearly, $\bS'$ is a torus of $\bG'$. Moreover, $|\bS^{F^m}| = 
|{\bS'}^{{F'}^m}|$ for all positive integers~$m$ (see
\cite[Proposition~$1.4.13$(c)]{GeMa}), implying the first claim.
By definition, $\bL = C_{\bG}( \bS )$ for some $\Phi_e$-torus~$\bS$ of~$\bG$.
As $\bL' = \nu( \bL ) = C_{\bG'}( \nu( \bS ) )$ by 
\cite[Corollary~$2$ to Proposition~$11.14$]{Borel},
the second assertion follows from the first one. 

(c) As~$\bL'$ is a Levi subgroup of~$\bG'$, we have 
$\bL' = C_{\bG'}( Z( \bL' )^\circ )$ by~\cite[Corollary~$14.19$]{Borel}. 
Put $\bL := (\nu^{-1}( \bL' ))^\circ$.
Then~$\bL$ is an $F$-stable closed, connected subgroup of~$\bG$ and 
$\nu( \bL ) = \bL'$. 
It remains to show that~$\bL$ is a Levi subgroup of~$\bG$. It follows from 
\cite[$11.14$(1)]{Borel}, that the unipotent radical of~$\bL$ is sent to 
the unipotent radical of~$\bL'$, which is trivial, as~$\bL'$ is reductive.
This implies that the unipotent radical of~$\bL$ is trivial, since the kernel 
of~$\nu$ is finite. Hence~$\bL$ is reductive.  Now, put $\bZ := Z( \bL )^\circ$. 
Then 
$$\nu( \bZ ) = \nu( Z( \bL )^\circ ) = ( \nu( Z( \bL ) ) )^\circ  
= Z( \bL' )^\circ,$$ where the second equality arises from 
\cite[Proposition~$7.4$.B(c)]{Hum75}, and the third one from 
\cite[$1.3.10$(c)]{GeMa}. Hence 
$$\nu( C_{\bG}( \bZ ) ) = C_{\bG'}( \nu( \bZ ) ) = \bL',$$ 
by \cite[Corollary~$2$ to Proposition~$11.14$]{Borel}. In
particular, we have $C_{\bG}( \bZ ) \unlhd \nu^{-1}( \bL' )$. As $C_{\bG}( \bZ )$ is 
connected, we obtain $C_{\bG}( \bZ ) \leq \bL$, and thus $\bL = C_{\bG}( \bZ )$. 
This implies that~$\bL$ is a Levi subgroup of~$\bG$.

(d) Let~$\bT$ denote a maximally split torus of~$\bG$. Then $\bT' := \nu(\bT)$
is a maximally split torus of~$\bG'$. Moreover,~$\nu$ induces an isomorphism
$\Phi: \mathbb{R} \otimes_{\mathbb{Z}} X( \bT' ) \rightarrow 
\mathbb{R} \otimes_{\mathbb{Z}} X( \bT )$ such that
$F \circ \Phi = \Phi \circ F'$; see, e.g.\ \cite[Theorem~$2.4.8$]{DiMi2}.
Thus~$\Phi$ maps the $F'$-eigenspaces on
$\mathbb{R} \otimes_{\mathbb{Z}} X( \bT' )$ to the $F$-eigenspaces on
$\mathbb{R} \otimes_{\mathbb{Z}} X( \bT )$, which proves our claim.

(e) We have $\nu( Z( \bG ) ) = Z( \bG' )$ by \cite[$1.3.10$(c)]{GeMa}.
In turn, $\nu( Z^{\circ}( \bG ) ) = Z^{\circ}( \bG' )$. Since
$\nu_{Z^{\circ}( \bG )} : Z^{\circ}( \bG ) \rightarrow Z^{\circ}( \bG' )$ is an
isogeny by~(a), we get $|Z^{\circ}( \bG )^F| = |Z^{\circ}( \bG' )^{F'}|$ from
\cite[Proposition $1.4.13$(c)]{GeMa}. As $Z^{\circ}( \bG )^F$ is a subgroup
of $Z( \bG )^F$ and $Z( \bG )^F = Z( G )$ by \cite[Proposition $3.6.8$]{C2},
we obtain our assertion.

(f) For the first statement see \cite[Proposition~$2.3.15$]{GeMa}. 
This bijection preserves the degrees of the characters and thus $e$-cuspidality 
by \cite[Proposition~$2.9$]{BrouMaMi}.
\end{proof}

\addtocounter{subsubsection}{1}
\subsubsection{Basic sets and $p$-rational characters}
\label{BasicSetsPRationalCharacters}
Recall that~$p$ is a prime with $p \neq r$, the characteristic of~$\mathbb{F}$.
For a semisimple $p'$-element $s \in G^*$ we define $\mathcal{E}_p( G, s )$ as 
in \cite[p.~$57$]{BrouMi}. Then $\mathcal{E}_p( G, s )$ is the set of characters 
of a union of $p$-blocks of~$G$; see \cite[Th{\'e}or{\`e}me~$2.2$]{BrouMi}.

The following lemma on $p$-rational characters in a block will be useful. Recall 
the definition of the Galois group~$\mathcal{G}$ introduced prior to
Lemma~\ref{IdentifyingTheNonExceptional}. For the concept of a basic set 
see~\cite[Subsection~$1.1$]{GeHi1}.

\addtocounter{thm}{1}
\begin{lem}
\label{pRationalCharacters}
Let $s \in G^*$ denote a semisimple $p'$-element.

{\rm (a)} Suppose that the pairs $(\bT^*,s)$ and $(\bT,\theta)$ are in duality.
Then the order of~$\theta$ is coprime to~$p$ and $R_{\bT}^{\bG}( \theta )$ is
$p$-rational. Moreover, $\mathcal{G}$ stabilizes $\mathcal{E}( G, s )$ as a set.

{\rm (b)} Let~$\bB$ be a $p$-block of~$G$ such that 
$\Irr( \bB ) \subseteq \mathcal{E}_p( G, s )$. If the restrictions of the 
characters of $\Irr( \bB ) \cap \mathcal{E}( G, s )$ to the $p$-regular classes
form a basic set for~$\bB$, every character of 
$\Irr( \bB ) \cap \mathcal{E}( G, s )$ is $p$-rational.
\end{lem}
\begin{proof}
(a) The first assertion follows from the reasoning of 
\cite[\textit{Remark}~$2.5.15$]{GeMa}, the second from the character formula 
\cite[Theorem~$7.2.8$]{C2}, and the third is contained in 
\cite[Proposition~$3.3.15$]{GeMa}.

(b) Two elements of $\Irr(G)$ in a $\mathcal{G}$-orbit have the same
restrictions to the $p$-regular classes of~$G$. Hence~$\mathcal{G}$ fixes
$\Irr(\bB)$ and therefore also $\Irr( \bB ) \cap \mathcal{E}( G, s )$ by~(a).
The claim then follows from our hypothesis.
\end{proof}

By \cite[Theorem~A]{GeII}, the hypothesis of Lemma~\ref{pRationalCharacters}(b)
is in particular satisfied if~$p$ is good for~$\bG$ and~$p$ does not divide 
$|Z( \bG )/Z^{\circ}( \bG )|$. The group $G = G_2( q )$ has unipotent characters, 
i.e.\ elements of $\mathcal{E}( G, 1 )$, that are not $3$-rational. Thus the
conclusion of Lemma~\ref{pRationalCharacters}(b) does not hold in general. 

\addtocounter{subsubsection}{1}
\subsubsection{Regular blocks} 
\label{RegularBlocksSection}
Let us keep the notation introduced in~\ref{BasicSetsPRationalCharacters}. In 
particular, $p$ is a prime distinct from~$r$, the characteristic 
of~$\mathbb{F}$. Lemma~\ref{SignAndOmegaInvariant} is especially useful in the 
context of {regular blocks}. Recall that a semisimple element $s \in \bG$ is 
called regular, if $C^{\circ}_{\bG}( s )$ is a maximal torus. Clearly, a regular
semisimple element is non-trivial, unless~$\bG$ is a torus. If we want to 
emphasize the underlying group, we also say that~$s$ is regular in~$\bG$, or 
in~$G$, if $s \in G$ and~$\bG$ is only given implicitly.

\addtocounter{thm}{1}
\begin{defn}
\label{RegularBlocksDefinition}
{\rm
(a) Let $s \in G$ be semisimple. We call~$s$ \textit{strictly regular}, if
$C_G( s )$ is a maximal torus of~$G$; in other words, if~$s$ is regular 
in~$\bG$ and $A_{\bG}(s)^F = 1$ (for the notion $A_{\bG}(s)$ see
\ref{ComponentGroups}).

(b) Let $s \in G^*$ be a semisimple $p'$-element and let $\bB$ be a $p$-block 
of~$G$ with $\Irr( \bB ) \subseteq \mathcal{E}_p( G, s )$. Suppose that 
$C^{\circ}_{\bG^*}( s ) = \bT^*$ for some maximal torus~$\bT^*$ of~$\bG^*$. 
Let~$\bT$ be an $F$-stable maximal torus of~$\bG$ dual to~$\bT^*$.
We then call~$\bB$ \textit{regular} with respect to~$T$; if, in addition,
$C_{G^*}( s ) = T^*$, we call~$\bB$ \textit{strictly regular}
with respect to~$T$.
}\hfill{$\Box$}
\end{defn}
Notice that if centralizers of semisimple elements in~$\bG^*$ are connected,
the notions of regular and strictly regular blocks coincide. Since, in the 
notation of Definition~\ref{RegularBlocksDefinition}(b), the torus~$\bT^*$ is 
uniquely determined by~$s$, the torus~$T$ is determined by~$s$ up to conjugation 
in~$G$. The reason for the attribute ``with respect to~$T$'' in the above 
definition will become clear in Lemma~\ref{GeneralizedRegularBlocks} and 
Corollary~\ref{StrictlyRegularBlocksCor} below.

Before we prove these results, we investigate regular elements under regular 
embeddings. 
Let $i : \bG \rightarrow \tilde{\bG}$ denote a regular embedding. Then there
is a reductive group $\tilde{\bG}^*$, dual to~$\bG^*$, an epimorphism
$i^* : \tilde{\bG}^* \rightarrow \bG^*$, and a compatible Steinberg morphism 
of~$\tilde{\bG}^*$, also denoted by~$F$; see, e.g.\ \cite[2.B, 2.D]{CeBo2}.
Recall that if~$\bT$ is an $F$-stable maximal torus of~$\bG$ and $\theta \in
\Irr(T)$, then~$\theta$ is\textit{ in general position}, if and only if the 
stabilizer~$N_G( \bT, \theta )$ of~$\theta$ in $N_G( \bT )$ equals~$T$. This is 
the case if and only if 
$\varepsilon_{\bG}\varepsilon_{\bT}R_{\bT}^{\bG}( \theta )$ is an irreducible 
character of~$G$; \cite[Corollary~$2.2.9$(a)]{GeMa}.

\begin{lem}
\label{RegularElementsAndRegularEmbeddings}
Let $\bT^*$ be an $F$-stable maximal torus of~$\bG^*$, and let $s \in T^*$ be a 
regular element, so that $C^{\circ}_{\bG^*}( s ) = \bT^*$. Choose an element 
$\tilde{s} \in \tilde{G}^*$ with $i^*( \tilde{s} ) = s$. Then $\tilde{s}$ is 
regular, with $C_{\tilde{\bG}^*}( \tilde{s} ) = \tilde{\bT}^*$,
where $\tilde{\bT}^* := {(i^*)}^{-1}( \bT^* )$.

Let~$\bT$ be an $F$-stable maximal torus of~$\bG$ and let $\theta \in \Irr(T)$ 
be such that the pairs $(\bT^*,s)$ and $(\bT, \theta)$ are in duality. Put 
$\tilde{\bT} := \bT Z( \tilde{\bG} )$. Then there is an extension
$\tilde{\theta}$ of~$\theta$ to $\tilde{T}$ such that 
$(\tilde{\bT}^*, \tilde{s} )$ and $(\tilde{\bT}, \tilde{\theta})$ are in 
duality.

Put $\chi :=
\varepsilon_{\bG}\varepsilon_{\bT}R_{\bT}^{\bG}( \theta )$ and
$\tilde{\chi} := 
\varepsilon_{\tilde{\bG}}\varepsilon_{\tilde{\bT}}
R_{\tilde{\bT}}^{\tilde{\bG}}( \tilde{\theta} )$.

We have $\mathcal{E}( \tilde{G}, \tilde{s} ) = \{ \tilde{ \chi } \}$ 
and
$$\chi = \Res^{\tilde{G}}_{G}( \tilde{ \chi } ) = \sum_{i=1}^{m} \chi_i$$
with pairwise distinct irreducible characters $\chi_i$, $i = 1, \ldots , m$ 
of~$G$. Moreover, $m = |A_{\bG^*}( s )^F|$ and
$\mathcal{E}( G, s ) = \{ \chi_1, \ldots , \chi_m \}$.

Finally,~$\theta$ is in general position, if and only if~$s$ is strictly 
regular.
\end{lem}
\begin{proof}
As $Z( \tilde{\bG} )$ is connected, $C_{\tilde{\bG}^*}( \tilde{s} )$ is 
connected by a theorem of Steinberg; see \cite[Theorem~$4.5.9$]{C2}.
As $i^*( C_{\tilde{\bG}^*}( \tilde{s} ) ) = C^{\circ}_{\bG^*}( s ) = \bT^*$ 
(see, e.g.\ \cite[p.~$36$]{CeBo2}), we obtain $C_{\tilde{\bG}^*}( \tilde{s} ) 
= {(i^*)}^{-1}( \bT^* ) = \tilde{\bT}^*$.

For the second statement see \cite[Lemme~$9.3$(b)]{CeBo2}. By Lusztig's Jordan
decomposition of characters we get $\mathcal{E}( \tilde{G}, \tilde{s} ) = 
\{ \tilde{\chi} \}$. Now $\chi = \Res^{\tilde{G}}_G( \tilde{\chi} )$ by
\cite[Proposition~$10.10$]{CeBo2}. By~\cite[Section~$10$]{Lu}, the
restriction $\Res^{\tilde{G}}_G( \tilde{\chi} )$ is multiplicity free, and its
constituents are in bijection with the unipotent characters of $C_{G^*}( s )$.
As $A_{\bG^*}( s )^F$ is abelian and $C^{\circ}_{\bG^*}( s )$ is a torus, the 
number of unipotent characters of $C_{G^*}( s )$ equals $|A_{\bG^*}( s )^F|$. 
The statement about $\mathcal{E}( G, s )$ follows from 
\cite[Proposition~$11.7$]{CeBo2}.  The last statement follows from the previous 
ones.
\end{proof}
We next investigate regular blocks and their defect groups.

\begin{lem}
\label{GeneralizedRegularBlocks}
Assume the notation and hypotheses of 
{\rm Lemma~\ref{RegularElementsAndRegularEmbeddings}}. Assume in addition 
that~$s$ is a $p'$-element, and choose an inverse image~$\tilde{s}$ 
under~$i^*$ as a $p'$-element as well. 

Then $p \nmid m$ and $N_G( \bT, \theta )/T \cong A_{\bG^*}( s )^F$.

Let $\tilde{\bB}$ be the $p$-block of~$\tilde{G}$ containing~$\tilde{\chi}$,
and let~$\bB$ be a block of~$G$ covered by~$\tilde{\bB}$. Then $\Irr( \bB )
\subseteq \mathcal{E}_p( G, s )$, the Sylow 
$p$-subgroup $\tilde{D}$ of $\tilde{T}$ is a defect group of~$\tilde{\bB}$, and 
$D := \tilde{D} \cap G$ is a defect group of~$\bB$. Also, $\Irr( \bB ) \cap
\mathcal{E}( G, s )$ is a non-empty subset of 
$\{ \chi_i \mid 1 \leq i \leq m \}$. 

Suppose that $t \in D$ is a $p$-element such that $C_{\bG}( t )$ is a regular
subgroup of~$\bG$.
Then $\chi_i( t ) = \chi_j( t )$ for all $1 \leq i, j \leq m$. In particular,
$$\chi_i(1) = \frac{1}{m}[G^*\colon\!T^*]_{r'}$$
for all $1 \leq i \leq m$.
\end{lem}

\begin{proof}
As~$s$ has $p'$-order, $p \nmid |A_{\bG^*}( s )^F|$; see \ref{ComponentGroups}.
Clearly, $\Irr( \bB ) \subseteq \mathcal{E}_p( G, s )$, since the irreducible 
constituents of~$\chi$ lie in $\mathcal{E}( G, s )$ by 
Lemma \ref{RegularElementsAndRegularEmbeddings},

Let $\mathbf{e}_s^G$ and $\mathbf{e}_s^T$ denote the central idempotents 
of $\mathcal{O}G$, respectively $\mathcal{O}T$, corresponding to
$\mathcal{E}_p( G, s )$, respectively $\mathcal{E}_p( T, s )$. Notice
that $\theta$ is the unique element in $\mathcal{E}( T, s )$, and that
the characters of $\mathcal{E}_p( T, s )$ belong to a single block of $kT$. 
The assertion $N_G( \bT, \theta )/T \cong A_{\bG^*}( s )^F$ thus follows from 
\cite[$(7.1)$]{BoDaRo}.

The fact that~$\tilde{D}$ is a defect group of~$\tilde{\bB}$ is contained 
in~\cite[Th{\'e}o\-r{\`e}\-me~$3.1$]{brouMo}.
As $p \nmid [N_G( \bT, \theta )\colon\!T]$, the blocks of~$N_G( \bT, \theta )$
covering $\mathbf{e}_s^T\mathcal{O}T$ have the Sylow $p$-subgroup 
$\tilde{D} \cap G$ of~$T$ as 
defect group. The corresponding assertion for $\mathbf{e}_s^G\mathcal{O}G$ 
follows from \cite[Theorem~$7.7$]{BoDaRo}.

Recall that $\mathcal{E}( G, s ) = \{ \chi_i \mid 1 \leq i \leq m \}$ by 
Lemma~\ref{RegularElementsAndRegularEmbeddings}.
Since $\Irr( \bB ) \subseteq \mathcal{E}_p( G, s )$ and~$s$ is a $p'$-element,
we have $\Irr( \bB ) \cap \mathcal{E}( G, s ) \neq \emptyset$ by 
\cite[Theorem~$3.1$]{gelfgraev}.

As~$C_{\bG}( t )$ is a 
regular subgroup of~$\bG$ by hypothesis, the $G$-conjugacy class of~$t$
equals the $\tilde{G}$-conjugacy class of~$t$; see 
Lemma~\ref{RegularEmbeddingsAndConjugacyClasses}(b). This implies the claim
on the values of $\chi_i(t)$, since the $\chi_i$, $1 \leq i \leq m$, are
conjugate under the action of~$\tilde{G}$. The final claim follows by choosing
$t = 1$.
\end{proof}

By Lemma~\ref{RegularElementsAndRegularEmbeddings}, every regular block arises
as one of the blocks~$\bB$ considered in Lemma~\ref{GeneralizedRegularBlocks}.
Regular blocks have first been introduced and investigated by Brou{\'e}
in~\cite{brouMo}, under the more restrictive condition that $C_{\bG^*}( s )$, in
the notation of Lemma~\ref{GeneralizedRegularBlocks}, is a torus. Under this 
condition we have $N_G( \bT, \theta ) = T$; in other words, 
$\theta \in \Irr( T )$ is in general position. Moreover, $m = 1$. As our 
applications are mostly concerned with the latter case, we formulate the
relevant results in a corollary. In the following, we will identify the
elements of $\Irr(T)$ of $p$-power order with their restrictions to
the Sylow $p$-subgroup~$D$ of~$T$.

\begin{cor}
\label{StrictlyRegularBlocksCor}
Let the hypotheses and notation be as in 
{\rm Lemma \ref{GeneralizedRegularBlocks}.} Assume in addition that $m = 1$,
i.e.\ $A_{\bG^*}( s )^F = 1$, so that~$\bB$ is strictly regular. 

Then $\theta \in \Irr(T)$ is in general position and $\chi \in \Irr(G)$, so that 
$\mathcal{E}( G, s ) = \{ \chi \}$. The Sylow $p$-subgroup~$D$ of~$T$ is a 
defect group of~$\bB$ and~$\bB$ is the unique block of~$G$ covered 
by~$\tilde{\bB}$.

Also, $\Irr( \bB ) = \{ 
\varepsilon_{\mathbf{G}}\varepsilon_{\mathbf{T}}
R_{\mathbf{T}}^{\mathbf{G}}( \lambda\theta ) \mid \lambda \in \Irr(D) \}$,
and $|\Irr( \bB )| = |D|$. In particular~$\bB$ has a unique irreducible Brauer 
character. If~$D$ is cyclic,~$\chi$ is the non-exceptional character of~$\bB$.
\end{cor}
\begin{proof}
The statements, except those in the last paragraph, immediately follow from 
Lemma~\ref{GeneralizedRegularBlocks}. Let us prove the final statements.

Clearly, $\lambda\theta$ is in general position for every $\lambda \in \Irr(D)$. 
We conclude that $\varepsilon_{\mathbf{G}}\varepsilon_{\mathbf{T}}
R_{\mathbf{T}}^{\mathbf{G}}( \lambda\theta ) \in \Irr( G )$ for every
$\lambda \in \Irr(D)$ and that
$$|\{ 
\varepsilon_{\mathbf{G}}\varepsilon_{\mathbf{T}}
R_{\mathbf{T}}^{\mathbf{G}}( \lambda\theta ) \mid \lambda \in \Irr(D) \}| 
= |D|.$$ 
It follows from \cite[Corollary~$7.3.5$]{DiMi2}, that restriction of 
class functions to the $p$-regular 
elements commutes with the Lusztig map. Hence the elements of the above set all 
have the same restriction to the $p$-regular elements, and all of them lie in 
$\Irr( \bB )$. Moreover,~$\chi$ is $p$-rational by 
Lemma~\ref{pRationalCharacters}(b).

In the notation of Lemma~\ref{RegularElementsAndRegularEmbeddings},
we have $C_{G^*}( ts ) = T^*$ for every $p$-element $t \in T^*$.
It follows that $|\mathcal{E}_p( G, s )| = |D|$. If~$D$ is cyclic,~$\chi$ is the
non-exceptional character of~$\bB$ by Lemma~\ref{IdentifyingTheNonExceptional}.
We have proved all our assertions.
\end{proof}

The last statement of Corollary~\ref{StrictlyRegularBlocksCor} can be 
generalized to the situation when~$\bB$ in 
Lemma~\ref{GeneralizedRegularBlocks} is cyclic and nilpotent. Under these 
hypotheses, the non-exceptional character in~$\bB$ is the unique character in  
$\Irr( \bB ) \cap \mathcal{E}( G, s )$. In view of 
Lemma~\ref{IdentifyingTheNonExceptional}, this is a consequence of the 
following, slightly more general, result. Due to its hypothesis, we cannot use
Lemma~\ref{pRationalCharacters}(b) right away.

\begin{lem}
\label{StrictlyRegularBlocksCorCyclic}
Keep the hypotheses and notation of 
{\rm Lemma \ref{GeneralizedRegularBlocks}}. Suppose that~$\Irr( \bB )$ contains
a unique $p$-rational character~$\psi_1$, and that~$\psi_1^{\circ}$ is the 
unique irreducible Brauer character of~$\bB$. (Here,~$\psi_1^{\circ}$ denotes 
the restriction of~$\psi_1$ to the $p$-regular elements of~$G$.)

Then the elements $\chi_1, \ldots , \chi_m$ of $\mathcal{E}( G, s )$ lie in~$m$ 
distinct $p$-blocks of $G$, and $\chi_i$ is the unique $p$-rational character 
in its block for $1 \leq i \leq m$.
\end{lem}
\begin{proof}
Recall that $\tilde{\bB}$ is a strictly regular block of~$\tilde{G}$ 
covering~$\bB$. By Corollary~\ref{StrictlyRegularBlocksCor}, applied with 
$G = \tilde{G}$ and $\bB = \tilde{\bB}$, all irreducible characters 
of~$\tilde{\bB}$ have the same degree. Let 
$\tilde{\psi} \in \Irr( \tilde{\bB} )$ such that 
$\psi := \Res_G^{\tilde{G}}( \tilde{\psi} )$ contains~$\psi_1$ as a 
constituent. Then $\psi = \sum_{i = 1}^l \psi_i$ with pairwise 
distinct irreducible characters $\psi_i$, $1 \leq i \leq l$; 
see~\cite[Section~$10$]{Lu}. For $1 \leq i \leq l$, let~$\bB_i$ denote the
block of~$G$ containing~$\psi_i$. Then the blocks $\bB = \bB_1, \ldots , \bB_l$
are pairwise distinct, since conjugation of characters preserves their 
$p$-rationality. Moreover, $\bB_1, \ldots , \bB_l$ are all the blocks of~$G$
covered by~$\tilde{\bB}$.

Choose the notation such that $\chi_1 \in \Irr( \bB ) = \Irr( \bB_1 )$. Then 
$\psi_1( 1 ) = \psi_1^{\circ}(1) \leq \chi_1(1)$. Since every block covered 
by~$\tilde{\bB}$ contains at least one of the characters 
$\chi_1, \ldots , \chi_m$ by 
Lemma~\ref{GeneralizedRegularBlocks}, we have $l \leq m$. It follows that
$\psi(1) = l\psi_1(1) \leq m\chi_1(1) = \chi(1) = \psi(1)$, and so $l = m$.
Thus each of the blocks $\bB_1, \ldots , \bB_m$ contains exactly one of the
characters $\chi_1, \ldots , \chi_m$.

Since $\chi = \varepsilon_{\bG}\varepsilon_{\bT}R_{\bT}^{\bG}( \theta )$ is 
$p$-rational by Lemma~\ref{pRationalCharacters}(a), the Galois 
group~$\mathcal{G}$ permutes the constituents $\chi_1, \ldots , \chi_m$ 
of~$\chi$. As $\mathcal{G}$-conjugate characters lie in the same 
block,~$\mathcal{G}$ fixes~$\chi_1$, which then is $p$-rational. This completes 
the proof.
\end{proof}

\section{Reduction theorems for classical groups}

The aim in this section is to reduce the computation of the source algebra 
equivalence classes of cyclic blocks of classical groups for the primes excluded 
in \cite[Proposition~$6.5$(a)]{HL24} to a corresponding problem for the general 
linear and unitary groups. Throughout this section, we fix a prime power~$q$, an 
odd prime~$p$ not dividing~$q$, and an integer $n \geq 1$. An algebraic closure
of the finite field with~$q$ elements is denoted by~$\mathbb{F}$.

\subsection{The relevant classical groups}
\label{TheRelevantClassicalGroups}
It is convenient to study the classical groups via their natural 
representations. For the notions and concepts with respect to classical groups, 
see~\cite{Taylor}. For the notions regarding Spin groups, 
see~\cite{Jac80,Gro02}.

Let~$V$ be an $n$-dimensional vector space over the finite 
field~$\mathbb{F}_{q^\delta}$, where $\delta \in \{ 1, 2 \}$. We assume that~$V$
is equipped with a form~$\kappa$ of one of the following types: a non-degenerate
hermitian form, a non-degenerate symplectic form, a non-degenerate quadratic form 
or the zero-form. If~$\kappa$ is hermitian, we let $\delta = 2$; in all other 
cases, $\delta = 1$. If~$\kappa$ is a quadratic form and $\dim(V)$ is odd, we 
assume that~$q$ is odd. By $I( V, \kappa )$ we denote the full isometry group 
of $(V, \kappa )$, and by~$\Omega( V, \kappa )$ the commutator subgroup 
of~$I( V, \kappa )$. If~$\kappa$ is a quadratic form, we write 
$\Spin( V, \kappa )$ for the corresponding spin group. (Notice that 
$\Spin( V, \kappa ) = \Omega( V, \kappa )$ in case~$q$ is even.) If~$\kappa$ is
the zero form or hermitian, $I( V, \kappa ) \cong \GL_n( q )$, 
respectively~$\GU_n( q )$ and $\Omega( V, \kappa ) \cong \SL_n( q )$, 
respectively $\SU_n( q )$, unless $(n,q) = (2,2)$. If~$\kappa$ is a symplectic 
form, $I( V, \kappa ) = \Omega( V, \kappa ) \cong \Sp_n( q )$, unless $(n,q)$ is
one of $(2,2), (2,3), (4,2)$. If~$\kappa$ is a quadratic form, the 
index of $\Omega( V, \kappa )$ in $I( V, \kappa )$ is a power of~$2$; see
\cite[Section~$11$]{Taylor}. Finally, $G( V, \kappa )$ denotes one of the 
following groups:
\begin{equation}
\label{ClassicalGroups}
G( V, \kappa ) := \left\{
       \begin{array}{ll}
               \Spin( V, \kappa ), & \text{if $\kappa$ is a quadratic form}, \\
               I(V, \kappa ), & \text{otherwise}. 
       \end{array}
       \right.
\end{equation}
We also consider the homomorphism
$\nu : G( V, \kappa ) \rightarrow I( V, \kappa )$, defined as follows.
If $\kappa$ is a quadratic form, $\nu$ is the \textit{vector representation of} 
$\Spin( V, \kappa )$, whose image equals~$\Omega( V, \kappa )$. In the other 
cases,~$\nu$ is the identity map. Put $Z := \ker(\nu)$. Then~$Z$ is trivial 
unless $\kappa$ is a quadratic form and $q$ is odd, in which case~$Z$ has 
order~$2$. This notation is used until the end of this section.

\subsection{Minimal polynomials}
\label{MinimalPolynomials}

We also need some notation concerning polynomials. 
Let~$\Delta$ be a monic irreducible polynomial over~$\mathbb{F}_{q^\delta}$.
If $\delta = 1$, we write $\Delta^*$ for the monic, irreducible polynomial 
over~$\mathbb{F}_{q}$ with the property: $\zeta \in \mathbb{F}$ is a root 
of~$\Delta$, if and only if $\zeta^{-1}$ is a root of~$\Delta^*$. If 
$\delta = 2$, we write $\Delta^\dagger$ for the monic, irreducible polynomial 
over~$\mathbb{F}_{q^2}$ with the property: $\zeta \in \mathbb{F}$ is a root 
of~$\Delta$, if and only if $\zeta^{-q}$ is a root of~$\Delta^\dagger$. 

\begin{lem}
\label{DegreesOfPolynomials}
Let~$\Delta$ be a monic irreducible polynomial over~$\mathbb{F}_{q}$ of 
degree~$e$, and let $\zeta \in \mathbb{F}$ be a root of~$\Delta$ with 
$|\zeta| = p$.

Then~$e$ is the order of~$q$ modulo~$p$. If $e = 2d$ is even, then 
$p \mid q^d + 1$ and $\Delta = \Delta^*$. If $\Delta = \Delta^*$, then~$e$ 
is even. 
\end{lem}
\begin{proof}
As~$\Delta$ is irreducible, we have 
$[\mathbb{F}_q[\zeta]\colon\!\mathbb{F}_q] = e$. Hence $p \mid q^e - 1$ as 
$|\zeta| = p$. Suppose that $p \mid q^j - 1$ for some positive integer~$j$. Then 
$\mathbb{F}_{q^j}$ contains all elements of~$\mathbb{F}^*$ of order~$p$, and so
$\zeta \in \mathbb{F}_{q^j}$. It follows that $\mathbb{F}_{q^e} = 
\mathbb{F}_q[\zeta] \subseteq \mathbb{F}_{q^j}$, and hence $e \leq j$.

Suppose that $e = 2d$ is even. Then $p \mid (q^d -1 )(q^d + 1)$. As~$p$ is odd
and $2d$ is the order of $q$ modulo~$p$, we conclude that $p \mid q^d + 1$.
As $|\zeta| = p$, we get $\zeta^{q^d + 1} = 1$, and so 
$\zeta^{-1} = \zeta^{q^d}$. As $\zeta^{q^d}$ is a root of~$\Delta$, we get 
$\Delta = \Delta^*$. On the other hand, if $\Delta = \Delta^*$, then~$e$ is 
even. This concludes the proof.
\end{proof}

\begin{lem}
\label{UnitaryDelta}
Suppose that $p \mid q + 1$. 
Let $\Delta$ be a monic irreducible polynomial over~$\mathbb{F}_{q^2}$ of odd 
degree~$e$, and let $\zeta \in \mathbb{F}$ be a root of~$\Delta$ of $p$-power 
order. Then $\Delta = \Delta^{\dagger}$.
\end{lem}
\begin{proof}
The claim is true if $\zeta = 1$. Thus assume that $\zeta \neq 1$.
By hypothesis, $q \equiv - 1\,\,(\mbox{\rm mod}\,\,p)$, and thus 
$q^e \equiv - 1\,\,(\mbox{\rm mod}\,\,p)$. Now $\zeta^{q^{2e} - 1} = 1$, and 
thus $\zeta^{q^e - 1}$ or $\zeta^{q^e + 1} = 1$, as~$p$ is odd. As 
$\zeta \neq 1$, the former case would imply $p \mid q^e - 1$, which is 
impossible. Hence $\zeta^{q^e} = \zeta^{-1}$.  It follows that 
$\zeta^{-q} = \zeta^{q^{e + 1}}$. As $e + 1$ is even, 
$\zeta^{q^{e + 1}}$ is a root of~$\Delta$, which proves our claim.
\end{proof}

The next lemma is used to investigate centralizers of possible defect groups.
\begin{lem}
\label{ShapeOfMinimalPolynomial}
Let $G := G( V, \kappa )$ with $G( V, \kappa )$ as 
in~{\rm (\ref{ClassicalGroups})}. Let $D = \langle t \rangle \leq G$ denote a
non-trivial cyclic radical $p$-subgroup of~$G$. Write $\bar{t} := \nu( t )$, and 
assume that~$\bar{t}$ has no
non-trivial fixed vector on~$V$. Let $\bar{t}_1$ denote a power of $\bar{t}$ of 
order~$p$. Let~$\Gamma$ denote the minimal polynomial of either~$\bar{t}$ or
of $\bar{t}_1$ on~$V$. 

Then~$\Gamma$ has at most two irreducible factors. If~$\kappa$ is the zero 
form,~$\Gamma$ is irreducible. If~$\kappa$ is a hermitian form, either 
$\Gamma = \Gamma^\dagger$ is irreducible, or $\Gamma = \Delta\Delta^\dagger$ 
where $\Delta$ is a monic irreducible polynomial over $\mathbb{F}_{q^2}$ with 
$\Delta \neq \Delta^\dagger$. In the other cases, either $\Gamma = \Gamma^*$ is 
irreducible, or $\Gamma = \Delta\Delta^*$, where~$\Delta$ is a monic irreducible 
polynomial over $\mathbb{F}_{q}$ with $\Delta \neq \Delta^*$. If~$\Gamma$ has 
two irreducible factors, each of them occurs with the same multiplicity in the 
characteristic polynomial of~$\bar{t}$, respectively~$\bar{t}_1$ on~$V$.

Finally,~$\bar{t}_1$ has no non-zero fixed vector on~$V$.
\end{lem}
\begin{proof}
Suppose first that~$\Gamma$ is the minimal polynomial of~$\bar{t}$. Notice 
that~$\nu(G)$ is a normal subgroup of~$I(V,\kappa)$ of $2$-power index.
Thus $\bar{D} := \langle \bar{t} \rangle$ is a 
radical $p$-subgroup of~$I(V,\kappa)$ by \cite[Lemma~$2.2$(b)(c)]{HL24}, 
and hence $\bar{D} = O_p( C_{I(V,\kappa)}( \bar{D} ) )$ by 
\cite[Lemma~$2.1$]{HL24}. The description of $C_{I(V,\kappa)}( \bar{t} )$ in
\cite[Theorem~$6.1.2$]{FLO} implies our claims.

Now suppose that~$\Gamma$ is the minimal polynomial of~$\bar{t}_1$. If 
$\bar{t}_1 = \bar{t}^j$, the roots of~$\Gamma$ are the $j$th powers of the 
roots of the minimal polynomial of~$\bar{t}$. Suppose first that~$\kappa$ is 
not a hermitian form so that the ground field is~$\mathbb{F}_q$. Let~$\varphi$ 
denote the Frobenius morphism of~$\mathbb{F}$ raising every element to its~$q$th 
power. Let $\zeta \in \mathbb{F}$ denote a root of~$\Gamma$. Then~$\Gamma$ is 
irreducible, if and only if every root of~$\Gamma$ is conjugate under 
$\langle \varphi \rangle$ to $\zeta$. The fact that $\Gamma = \Gamma^*$ is 
irreducible, or $\Gamma = \Delta\Delta^*$, where~$\Delta$ is a monic irreducible 
polynomial over $\mathbb{F}_{q}$ with $\Delta \neq \Delta^*$, is equivalent to 
the fact that any other root of~$\Gamma$ is conjugate under 
$\langle \varphi \rangle$ 
to $\zeta$ or to $\zeta^{-1}$. Thus the asserted properties of~$\Gamma$ follow 
from the corresponding properties of the minimal polynomial of~$\bar{t}$.
An analogous argument works in the case when~$\kappa$ is a hermitian form, 
replacing~$\varphi$ by~$\varphi^2$ and $\zeta^{-1}$ by~$\zeta^{-q}$.

The claim about the multiplicity in the characteristic polynomial follows from 
$\bar{t}, \bar{t}_1 \in I(V,\kappa)$.

The last assertion follows from the previous ones, which imply that all 
non-trivial eigenvalues of~$\bar{t}$ have the same order.
\end{proof}

\subsection{A reduction lemma} Let~$\bB$ be a $p$-block of a group~$G(V,\kappa)$ 
with a non-trivial cyclic defect group~$D$. The following lemma reduces the
computation of the invariant $W(\bB)$ to the case where the fixed space 
of~$\nu(D)$ on~$V$ is trivial.

\begin{lem}
\label{NoEigenvalueIs1}
Let $G := G( V, \kappa )$ be one of the groups introduced 
in~{\rm (\ref{ClassicalGroups})} and let~$\bB$ be a $p$-block of~$G$ with a 
non-trivial cyclic defect group~$D$. Put $\bar{D} := \nu( D )$.

Let~$V^0$ be the fixed space of~$\bar{D}$ on~$V$, and let~$\kappa^0$ denote 
the restriction of~$\kappa$ to~$V^0$. Then~$(V^0,\kappa^0)$ is non-degenerate,
unless~$\kappa$ is the zero form. Let~$V'$ denote the orthogonal complement 
of~$V^0$ (or any complement in case~$\kappa$ is the zero form). Finally, let 
$\kappa'$ denote the restriction of $\kappa$ to~$V'$.

Then $G( V', \kappa' )$ embeds into~$G$, and~$D$ is contained in the image~$G'$ 
of this embedding. Moreover, there is a block~$\bB'$ of~$G'$ with defect 
group~$D$ such that $W( \bB ) \cong W( \bB' )$.
\end{lem}
\begin{proof}
Write $\bar{G} := I( V, \kappa )$. As~$D$ is a defect group of a $p$-block 
of~$G$, it is a radical $p$-subgroup of~$G$. Put~$\bar{D}_1 := \nu( D_1 )$, 
where~$D_1$ denotes the unique subgroup of~$D$ of order~$p$. Let~$t$ denote a 
generator of~$D$ and put $\bar{t} := \nu(t)$.

Suppose that~$\kappa$ is not the zero form. Then $(V^0, \kappa^0 )$ is 
non-degenerate, since the eigenspaces 
of~$\bar{t}$ on~$V$ for eigenvalues unequal to~$1$ are orthogonal to~$V^0$.
Identify $I^0 := I( V^0, \kappa^0 )$ and $I' := I( V', \kappa ' )$ with 
subgroups of~$\bar{G}$, and~$\bar{t}$ with an element of~$I'$ in the natural 
way.  

If $\kappa$ is a quadratic form and $q$ is odd, let 
$G^0 := \nu^{-1}( \Omega( V^0, \kappa^0 ))$ and 
$G' := \nu^{-1}( \Omega( V', \kappa' ))$. Then 
$G^0 \cong \Spin( V^0, \kappa^0 )$ and $G' \cong \Spin( V', \kappa' )$; see, 
e.g.\ \cite[Chapter V, Section~$4$]{Artin}. Otherwise, let 
$G^0 := G( V^0, \kappa^0 )$ and $G' := G( V', \kappa' )$. 

We claim that $D \leq G'$ and that $N_{G'}( D_1 )$ is a normal subgroup of 
$N_G( D_1 )$. We only prove these claims in case~$\kappa$ is a quadratic form.
The proof in the other cases is analogous but simpler. Assume that~$\kappa$ is a 
quadratic form and  write $\Omega^0 := \Omega( V^0, \kappa^0 )$ and $\Omega' :=
\Omega( V', \kappa' )$. Then $\bar{t} \in \Omega'$, as the index of $\Omega'$ 
in~$I'$ is a power of~$2$. Thus $t \in \nu^{-1}( \Omega' ) = G'$, which is our
first claim. As~$V^0$ is the fixed space of~$\bar{D_1}$ by 
Lemma~\ref{ShapeOfMinimalPolynomial}, we get
$\Omega^0 \times N_{\Omega'}( \bar{D}_1 ) \leq N_{\nu(G)}( \bar{D}_1 ) \leq 
I^0 \times N_{I'}( \bar{D}_1 )$. In particular, $N_{\Omega'}( \bar{D}_1 ) 
\unlhd N_{\nu(G)}( \bar{D}_1 )$. Now $\nu( N_{G}( D_1 ) ) = 
N_{\nu(G)}( \bar{D}_1 )$ by \cite[Lemma~$2.2$(a)]{HL24}. The kernel 
of~$\nu$ is contained in $N_{G}( D_1 )$, and hence 
$N_G( D_1 ) = \nu^{-1}( \nu( N_G( D_1 ) ) ) = \nu^{-1}( N_{\nu(G)}( \bar{D}_1 ))$. 
Thus $\nu^{-1}( N_{\Omega'}( \bar{D}_1 ) ) \unlhd N_{G}( D_1 )$. As
$\nu^{-1}( N_{\Omega'}( \bar{D}_1 ) ) = N_{G'}( D_1 )$, once more by 
\cite[Lemma~$2.2$(a)]{HL24}, we obtain our second claim.

Let~$\bb$ denote the 
Brauer correspondent of~$\bB$ in $N_G( D_1 )$ and let~$\bb'$ be a block 
of~$N_{G'}( D_1 )$ covered by~$\bb$. As 
$D \leq N_{G'}( D_1 ) \unlhd N_G( D_ 1)$, there is a defect group~$D'$ 
of~$\bb'$ conjugate to~$D$ in~$N_G( D_1 )$. We may thus assume that
$D' = D$, and let $\bB'$ denote the Brauer correspondent of $\bb'$ in~$G'$.
Then $\bB'$ has defect group~$D$ and we find
$W( \bB ) \cong W( \bb ) \cong W( \bb' ) \cong W( \bB' )$, where the middle 
isomorphism follows from \cite[Corollary~$4.4$]{HL24}. This completes our proof.
\end{proof}
 
\subsection{The underlying algebraic groups}
\label{UnderlyingAlgebraicGroups}

To continue, we also need to consider the algebraic groups and their Steinberg 
morphisms that give rise to the finite groups introduced in 
Subsection~\ref{TheRelevantClassicalGroups}. At the same time we specify the
degrees~$n$ relevant to our investigation.

Let~$\bG$ denote one of the following classical groups over~$\mathbb{F}$, namely
$\GL_n( \mathbb{F} )$, $\Sp_n( \mathbb{F} )$ with~$n$ even, or 
$\Spin_n( \mathbb{F} )$ with $q$ odd if~$n$ is odd. Then there is an 
$\mathbb{F}_q$-rational structure on~$\bG$ and a corresponding Frobenius 
morphism~$F$, such that $G = \bG^F \cong G(V, \kappa)$ in the notation 
of Subsection~\ref{TheRelevantClassicalGroups}. The various cases are displayed
in Table~\ref{ClassicalAlgebraicGroups}.
\begin{table}
\caption{\label{ClassicalAlgebraicGroups} Some classical algebraic groups}
$
\begin{array}{cclccc}\\ \hline\hline
\text{\rm Case} & \bG & \text{\rm Condition} & F & \bG^F & \kappa \rule[- 3pt]{0pt}{ 16pt} \\ \hline\hline
1 & \GL_n( \mathbb{F} ) & n \geq 1 & \begin{array}{c} F_{+1} \\ F_{-1} \end{array} & 
\begin{array}{c} \GL_n( q ) \\ \GU_n( q ) \end{array} &
\begin{array}{c} \text{\rm zero} \\ \text{\rm hermitian} \end{array} \rule[ 0pt]{0pt}{ 19pt} \\ \hline
2 & \Sp_n( \mathbb{F} ) & n \geq 4 \text{\rm\ even} & F & \Sp_n( q ) & 
\text{\rm symplectic} \rule[-3pt]{0pt}{ 13pt} \\ \hline
3 & \Spin_n( \mathbb{F} ) & n \geq 7 \text{\rm\ odd}, q \text{\rm\ odd} & 
F & \Spin_n( q ) & \text{\rm quadratic} \rule[-5pt]{0pt}{ 15pt} \\ \hline
4 & \Spin_n( \mathbb{F} ) & n \geq 8 \text{\rm\ even} & \begin{array}{c} F_{+1} \\ F_{-1} \end{array} & 
\begin{array}{c} \Spin^+_n( q ) \\ \Spin^-_n( q ) \end{array} &
\begin{array}{c} \text{\rm quadratic} \\ \text{\rm quadratic} \end{array} \rule[- 2pt]{0pt}{20pt} \\ \hline\hline
\end{array}
$
\end{table}

Let us now be more specific. For a positive integer~$m$, let~$J_m$ denote the 
$(m \times m)$-matrix with $1$'s along the anti-diagonal, and $0$'s elsewhere.
Let $\bV := \mathbb{F}^n$, the standard column space of dimension~$n$ 
over~$\mathbb{F}$. Gram matrices of bilinear forms on~$\bV$ are taken with 
respect to the standard basis. In Case~$1$, we put $\bG := \GL( \bV )$. In 
Case~$2$, we have $n = 2m$ even, and we let $\boldsymbol{\omega}$ denote the 
symplectic form on~$\bV$, whose Gram matrix equals 
$\left( \begin{array}{cc} 0 & J_m \\ -J_m & 0 \end{array} \right)$. In Case~$3$, 
we have $q$ odd, and we let $\boldsymbol{\omega}$ denote the quadratic form 
on~$\bV$, whose polar form has Gram matrix~$J_n$. In Case~$4$, we have $n = 2m$ 
even, and we let~$\boldsymbol{\omega}$ denote the quadratic form on~$\bV$ 
defined by
$$\boldsymbol{\omega}( (x_1, \ldots , x_n)^t ) = \sum_{i = 1}^m x_i x_{n-i+1}$$
for $(x_1, \ldots , x_n)^t \in \bV$. Then the polar form of~$\boldsymbol{\omega}$
has Gram matrix~$J_n$. In Cases~$2$--$4$, the forms~$\boldsymbol{\omega}$ are 
non-degenerate, and we write $I( \bV, \boldsymbol{\omega} )$ for their full 
isometry groups. In Case~$2$, we put $\bG := I( \bV, \boldsymbol{\omega} )$.

In Cases~$3$ and~$4$, we let~$\bG$ denote the spin group 
of~$\boldsymbol{\omega}$, and we have a homomorphism 
$\nu : \mathbf{G} \rightarrow I( \bV, \boldsymbol{\omega} )$ of algebraic 
groups, the vector representation of~$\mathbf{G}$. Its kernel has order 
$\gcd( 2, q - 1)$, its image $\bar{\mathbf{G}} := \SO_n( \mathbb{F} )$ is the 
connected component of~$I( \bV, \boldsymbol{\omega} )$. If~$q$ is odd, 
$\SO_n( \mathbb{F} ) = I( \bV, \boldsymbol{\omega} ) \cap \SL_n( \mathbb{F} )$. 
As the kernel of~$\nu$ consists of $F$-fixed points, there is a unique Steinberg 
morphism of~$\bar{\mathbf{G}}$, also denoted by~$F$, such that~$\nu$ is 
$F$-equivariant. Then $\bar{G} = \bar{\mathbf{G}}^F = \nu( \mathbf{G} )^F$ is a
special orthogonal group of the appropriate type, 
and $\nu( G ) = \nu( \mathbf{G}^F ) \cong \Omega(V, \kappa )$.
If~$n$ is even and~$q$ is odd, $[\bar{G}\colon\!\nu(G)] = 2$; 
otherwise, $\bar{G} = \nu(G)$. Finally, there is a Steinberg 
morphism~$F$ on~$I( \bV, \boldsymbol{\omega} )$, which restricts to the given 
Steinberg morphism~$F$ on $\bar{\bG}$ such that 
$I( \bV, \boldsymbol{\omega} )^F \cong I( V, \kappa )$.

\begin{lem}
\label{NonTypeAClassicalGroups}
Let $(\mathbf{G},F)$ be as in one of the Cases~$3$ or~$4$ of 
{\rm Table~\ref{ClassicalAlgebraicGroups}}. Write~$Z$ for the kernel of~$\nu$. 
Let $t \in G$ be a $p$-element and put $\bar{t} := \nu( t )$. Then there is a 
short exact sequence
\begin{equation}
\label{IsogenyOfCentralizers}
1 \longrightarrow Z \longrightarrow C_{\bG}( t ) \stackrel{\nu}{\longrightarrow} 
C_{\bar{\bG}}( \bar{t} ) \longrightarrow 1.
\end{equation}
In particular, $|C_G( t )| = |C_{\bar{G}}( \bar{t} )|$.
\end{lem}
\begin{proof}
As~$p$ is odd, $C_{\bar{\bG}}( \bar{t} )$ is connected by 
\cite[Corollary~$2.6$]{GeHi1}. The first claim thus follows from 
\cite[$1.3.10$(e)]{GeMa}, and the second one 
from \cite[Proposition~$1.4.13$(c)]{GeMa}. 
\end{proof}

\subsection{The symplectic and orthogonal groups} 
\label{TheSymplecticAndOrthogonalGroups} Here we investigate the 
centralizers of possible defect groups of cyclic $p$-blocks in the symplectic
and orthogonal groups, under the restrictions suggested by 
Lemma~\ref{NoEigenvalueIs1}. Let $(\bG,F)$ be as in Cases $2$--$4$ of 
Table~\ref{ClassicalAlgebraicGroups}, so that $\bG^F = G \cong G( V, \kappa )$ with 
$G( V, \kappa )$ as in~{\rm (\ref{ClassicalGroups})}, where~$\kappa$ is a 
symplectic form if~$\bG$ is as in Case~$2$ of 
Table~\ref{ClassicalAlgebraicGroups}, and a quadratic form, otherwise. 
If~$\kappa$ is a symplectic form, put $\bar{\bG} := \bG$, and let 
$\nu : \bG \rightarrow \bar{\bG}$ denote the identity map.
If~$\kappa$ is a quadratic form, let $\bar{\bG} = \SO_n( \mathbb{F} )$ denote 
the image of the vector representation~$\nu$ of~$\bG$. Let~$Z$ denote the kernel
of~$\nu$.

From now on, we will use the common $\varepsilon$-convention for 
the groups $\GL_n(q)$ and $\GU_n(q)$. Let $\varepsilon \in \{ 1, -1 \}$. Then 
$\GL^\varepsilon_n(q)$ denotes the general linear group $\GL_n(q)$, if 
$\varepsilon = 1$, and the general unitary group $\GU_n(q)$, if 
$\varepsilon = -1$.  Analogous conventions are used for $\SL^\varepsilon_n(q)$ 
and $\PSL^\varepsilon_n(q)$.

\begin{lem}
\label{NonTypeAClassicalGroups0}
Assume the notation and hypothesis introduced at the beginning of 
{\rm Subsection~\ref{TheSymplecticAndOrthogonalGroups}}.

Let $t \in G$ be a non-trivial $p$-element such that $\langle t \rangle$ is a 
radical $p$-subgroup of~$G$. Let~$t_1$ denote a power of~$t$ of order~$p$.

Put $\bar{t}_1 := \nu(t_1)$. Assume that $\bar{t}_1$ does not fix any
non-trivial vector of~$V$, and let~$\Gamma$ denote the minimal polynomial
of~$\bar{t}_1$. By {\rm Lemmas~\ref{ShapeOfMinimalPolynomial}} 
and~{\rm \ref{DegreesOfPolynomials}}, the degree 
of~$\Gamma$ is even, $2d$, say, and~$\Gamma$ has at most two irreducible 
constituents. Put $\varepsilon := -1$, if $\Gamma$ is irreducible, and 
$\varepsilon := 1$, otherwise.

Then $n = 2md$ for some positive integer~$m$, 
$$C_{\bar{\bG}}( \bar{t}_1 ) \cong 
\GL_m( \mathbb{F} ) \times \cdots \times \GL_m( \mathbb{F} ),$$
with~$d$ factors, and 
$$C_{\bar{G}}( \bar{t}_1 ) \cong \GL^\varepsilon_m( q^d ).$$
\end{lem}
\begin{proof}
Clearly, $n = 2md$, where~$m$ denotes the common multiplicity of the irreducible
factors of~$\Gamma$ in the characteristic polynomial of~$\bar{t}_1$. In 
particular,~$\bG$ is not as in Case~$3$ of Table~\ref{ClassicalAlgebraicGroups}. 

Let~$\zeta \in \mathbb{F}$ denote a root of~$\Gamma$, and write $\bV_{\zeta}$ 
for the corresponding eigenspace of~$\bar{t}_1$ on~$\bV$. Then $\bV_{\zeta}$ is 
$C_{I(\bV,\boldsymbol{\omega})}( \bar{t}_1 )$-invariant and 
$\dim( \bV_{\zeta} ) = m$. Let $\zeta'$ be a root of~$\Gamma$ with 
$\zeta' \neq \zeta^{-1}$. A straightforward calculation shows that 
$\bV_{\zeta}$ and $\bV_{\zeta'}$ are orthogonal with respect 
to~$\boldsymbol{\omega}$ in Case~$2$, respectively to the polar form 
of~$\boldsymbol{\omega}$ in Case~$4$. In particular,~$\bV_{\zeta}$ is totally 
isotropic since $\zeta \neq \pm1$, and $\bV_{\zeta} \oplus \bV_{\zeta^{-1}}$ is 
non-degenerate, as~$\boldsymbol{\omega}$ in Case~$2$, respectively the polar 
form of~$\boldsymbol{\omega}$ in Case~$4$, are non-degenerate.  
The restriction of $C_{I(\bV,\boldsymbol{\omega})}( \bar{t}_1 )$ to 
$\bV_{\zeta} \oplus \bV_{\zeta^{-1}}$ is isomorphic to $\GL_m( \mathbb{F} )$. 
Hence
$C_{I(\bV,\boldsymbol{\omega})}( \bar{t}_1 ) \cong 
\GL_m( \mathbb{F} ) \times \cdots \times \GL_m( \mathbb{F} )$ with~$d$ factors. 
In particular, $C_{I(\bV,\boldsymbol{\omega})}( \bar{t}_1 )$ is connected.

It follows that $C_{I(\bV,\boldsymbol{\omega})}( \bar{t}_1 ) \leq \bar{\bG}$ as
$\bar{\bG} = I(\bV,\boldsymbol{\omega})^{\circ}$. Thus
$C_{I(\bV,\boldsymbol{\omega})}( \bar{t}_1 ) = C_{\bar{\bG}}( \bar{t}_1 )$,
proving our first claim. Now $C_{\bar{G}}( \bar{t}_1 ) = 
C_{I(V,\omega)}( \bar{t}_1 )$, and by \cite[Theorem~$6.1.2$]{FLO}, we obtain
$C_{I(V,\omega)}( \bar{t}_1 ) \cong \GL^\varepsilon_m( q^d )$. This is our 
second claim.
\end{proof}

\begin{rem}
\label{RemNonTypeAClassicalGroups0}
{\rm
Assume the notation and hypotheses of Lemma~\ref{NonTypeAClassicalGroups0}.
If $\Gamma$ is irreducible, then~$2d$ is the order 
of~$q$ modulo~$p$, and $p \mid q^d + 1$. On the other hand, if $\Gamma = 
\Delta\Delta^*$ with $\Delta \neq \Delta^*$, then~$d$ is odd, and~$d$ is the 
order of~$q$ modulo~$p$. Thus $p \mid q^d - \varepsilon$.

Suppose that~$\kappa$ is a quadratic form. If $\Gamma = \Delta\Delta^*$ with
$\Delta \neq \Delta^*$, then $G = \Spin^+_n( q )$. On the other hand, if
$\Gamma$ is irreducible, then $G = \Spin^+_n( q )$ if~$m$ is even, and
$G = \Spin^-_n( q )$ if~$m$ is odd. These results are proved in
\cite[Lemmas $4.1$, $4.5$]{HM19}.
}\hfill{$\Box$}
\end{rem}

\subsection{Reduction to the general linear and unitary groups}
We now put our focus on the centralizers of elements of order~$p$, under the 
hypotheses of Lemma~\ref{NonTypeAClassicalGroups0}. The lemma below reduces 
the computation of the invariants $W( \bB )$ in such centralizers to the 
corresponding question in general linear and unitary groups. The groups~$\bG$
and $\bar{\bG}$ in Lemma~\ref{BlocksAndIsogeny} below play the roles of the 
groups $C_{\bG}( t_1 )$, respectively $C_{\bar{\bG}}( \bar{t}_1 )$ of 
Lemma~\ref{NonTypeAClassicalGroups0}. Notice that in the situation of 
Lemma~\ref{NonTypeAClassicalGroups0}, the groups~$C_{G}( t_1 )$ 
and~$C_{\bar{G}}( \bar{t}_1 )$ need not be isomorphic, although they have the 
same order by Lemma~\ref{NonTypeAClassicalGroups}. For example, suppose that 
$G = \Spin_n^+( q )$ with~$q$ odd and $4 \mid n$. Then~$Z( G )$ is
elementary abelian of order~$4$; see \cite[Table~$6.1.2$]{Gor}. In this case,
we cannot have $C_G( t_1 ) \cong C_{\bar{G}}( \bar{t}_1 )$, as 
$C_{\bar{G}}( \bar{t}_1 ) \cong \GL^\varepsilon_m( q^d )$ has a cyclic center.

Fix $\varepsilon \in \{ -1, 1 \}$. In view of 
Remark~\ref{RemNonTypeAClassicalGroups0}, replacing $q^d$ by~$q$, the hypothesis 
$p \mid q - \varepsilon$ of Lemma~\ref{BlocksAndIsogeny} below is justified.
This lemma is based mainly on \cite{CaEn99}. For a concise and in parts more
general account of these results see \cite[Sections~$2$,~$3$]{KeMa}.

\begin{lem}
\label{BlocksAndIsogeny}
Let~$\bG$ be a connected reductive algebraic group and let~$\bar{\bG}$ denote a 
direct product of groups isomorphic to $\GL_n( \mathbb{F} )$. Let~$F$ 
and~$\bar{F}$ be Frobenius morphisms of~$\bG$ and~$\bar{\bG}$, respectively, 
arising from $\mathbb{F}_q$-rational structures of~$\bG$ and~$\bar{\bG}$, 
respectively, such that
$\bar{G} = \bar{\bG}^{\bar{F}} \cong \GL_n^\varepsilon( q )$. 
Assume that $p \mid q - \varepsilon$.

Suppose that there is an isogeny $\nu : \bG \rightarrow \bar{\bG}$ with 
$\nu \circ F = \bar{F} \circ \nu$ and $|\!\ker(\nu)| \leq 2$.

Let $\bB$ be a $p$-block of~$G$ with a non-trivial cyclic defect group~$D$. Then 
there is an $F$-stable maximal torus~$\bT$ of~$\bG$, such that~$\bB$ is regular 
with respect to~$T$; see {\rm Definition~\ref{RegularBlocksDefinition}(b)}. By 
{\rm Lemma~\ref{GeneralizedRegularBlocks}},
we may thus assume that~$D$ is a Sylow $p$-subgroup of~$T$.

Put $\bar{\bT} := \nu( \bT )$. Then~$\bar{T}$ is cyclic of order 
$q^n - \varepsilon^n$.

Suppose that there exists a $p$-block~$\bar{\bB}$ of~$\bar{G}$, which is regular  
with respect to $\bar{T}$. Then $W( \bB ) \cong W( \bar{\bB} )$, if~$D$ 
and $\bar{D} = \nu(D)$ are identified.
\end{lem}
\begin{proof}
We will make use of the main results 
of~\cite{CaEn99}. Put $\bar{\bG}^* := \bar{\bG}$, and let~$\bG^*$ be a reductive 
group dual to~$\bG$. Let~$F$ be a Frobenius morphism of~$\bG^*$ such that the 
pairs $(\bG, F)$ and $(\bG^*,F)$ are in duality.
There is a dual isogeny $\nu^*: \bar{\bG}^* \rightarrow \bG^*$ satisfying
$F \circ \nu^* = \nu^* \circ \bar{F}$; see \cite[Proposition~$11.1.11$]{DiMi2}. 
By Lemma~\ref{RelativeRankAndIsogenies}(e) we have $\nu( Z^{\circ}( \bG ) ) = 
Z^{\circ}( \bar{\bG} ) = Z( \bar{\bG} ) = \nu( Z( \bG ) )$. Hence $Z( \bG ) = 
Z^{\circ}( \bG ) \cdot \ker( \nu )$, and thus 
$|Z( \bG )/Z^{\circ}( \bG )| \leq 2$. We also have $Z^{\circ}( \bG^* ) =
\nu^*( Z^{\circ}( \bar{\bG}^* ) ) = \nu^*( Z( \bar{\bG}^* ) ) = Z( \bG^* )$, 
where the last equality arises from \cite[$1.3.10$(c)]{GeMa}. Hence $Z( \bG^* )$ 
is connected. Thus $3 \in \Gamma( \bG, F) \cap \Gamma( \bar{\bG}, \bar{F})$ in 
the notation of \cite[Definition-Notation~$4.3$]{CaEn99}, so that 
\cite[Theorems~$4.1$,~$4.2$]{CaEn99} apply to $( \bG, F)$ and to 
$( \bar{\bG}, \bar{F})$ also in case $p = 3$; see 
\cite[Subsection~$5.2$]{CaEn99}. Notice that~$p$ is good for~$\bG$, as~$\bG$ is 
of type~$A$.

Let~$e$ denote the order of $q$ modulo~$p$. Thus $e = 1$ if $p \mid q - 1$,
i.e.\ if $\varepsilon = 1$, and $e = 2$, if $p \mid q + 1$, i.e.\ if 
$\varepsilon = -1$. According to \cite[Theorem]{CaEn99}, there is an
$e$-cuspidal pair $(\bL, \chi')$ associated to~$\bG$ which determines~$\bB$.
This means that~$\bL$ is an $e$-split Levi subgroup of~$\bG$, and $\chi'$ is 
an $e$-cuspidal irreducible character of~$L$. Moreover,~$\chi'$ lies in a
Lusztig series of~$L$ determined by a $p'$-element. By
\cite[Theorem~$4.2$]{CaEn99}, the pair $(\bL,\chi')$ satisfies the Jordan
criterion, i.e.\ Condition~(J) of \cite[Subsection~$1.3$]{CaEn99}. In the 
notation of~\cite{CaEn99}, we have $\bG_{\mathbf{a}} = \bG$ by our hypothesis 
on~$\bG$. Let~$Z$ be as defined in \cite[Definition-Notation~$4.3$ and 
Definition~$4.6$]{CaEn99}. Then $Z = O_p( T )$ for some $F$-stable maximal
torus~$\bT$ of~$\bL$; see \cite[Lemma~$4.8$]{CaEn99}. Moreover, 
by \cite[Definition~$4.9$]{CaEn99}, and \cite[Lemma~$4.16$]{CaEn99}, we may
assume that $D = Z \leq T$. (Notice that in the corresponding statement of
\cite[Lemma~$4.16$]{CaEn99} the attribute ``maximal'' is missing.) In 
particular, $O_p( Z( L ) ) \leq O_p( T ) = Z$, and so $O_p( Z( L ) )$ is 
cyclic.

Now $\bar{\bL} := \nu(\bL)$ is an $e$-split Levi subgroup of~$\bar{\bG}$ by
Lemma~\ref{RelativeRankAndIsogenies}(b). Furthermore, 
$O_p( Z( L ) ) \cong O_p( Z( \bar{L} ) )$, as 
$\nu( Z( \bL ) ) = Z( \bar{\bL} )$; see \cite[$1.3.10$(c)]{GeMa}. We also have
$\bar{L} \cong \GL_{n_1}^\varepsilon( q ) \times \cdots \times 
\GL_{n_c}^\varepsilon( q )$ with $\sum_{j= 1}^c n_j = n$. As 
$O_p( Z( \bar{L} ) )$ is cyclic, we get $c = 1$, thus $\bar{\bL} = \bar{\bG}$, 
and hence $\bL = \bG$. In particular,~$\chi'$ is an $e$-cuspidal character 
of~$G$, and $\chi' \in \Irr( \bB )$. 

Let $s \in G^*$ be a semisimple $p'$-element such that~$\chi'$ lies in the 
Lusztig series~$\mathcal{E}( G, s )$ and put $\bT^* := C^{\circ}_{\bG^*}( s )$. 
Then~$\bT^*$ is an $F$-stable Levi subgroup of~$\bG^*$, as the latter is a group 
of type~$A$. By Lemma~\ref{RelativeRankAndIsogenies}(c), there is
an $\bar{F}$-stable Levi subgroup~$\bar{\bT}^*$ of~$\bar{\bG}^*$ with 
$\nu^*( \bar{\bT}^* ) = \bT^*$.  

Condition~($\text{J}_2$) of \cite[Subsection~$1.3$]{CaEn99} implies 
that~$C_{G^*}^{\circ}(s)$ has an $e$-cuspidal unipotent character. 
Then~$\bar{T}^*$ has an $e$-cuspidal unipotent character by
Lemma~\ref{RelativeRankAndIsogenies}(f). In turn,~$\bar{\bT}^*$ is a maximal 
torus by \cite[Proposition~$2.9$]{BrouMaMi}. In particular,~$\bT^*$ is a
maximal torus.
Notice that~$\nu^*$
maps the Sylow $\Phi_e$-torus of~$\bar{\bT}^*$ to the Sylow $\Phi_e$-torus
of~$\bT^*$; see Lemma~\ref{RelativeRankAndIsogenies}(b). 
Condition~($\text{J}_1$) of \cite[Subsection~$1.3$]{CaEn99} and the
description of the maximal tori of $\GL_n^\varepsilon(q)$ now show
that~$\bar{T}^*$ is cyclic of order $q^n - \varepsilon^n$.

Let~$\bT$ denote an $F$-stable maximal torus of~$\bG$, dual to~$\bT^*$.
Then~$\bB$ is regular with respect to~$\bT$ by definition. Also, $\bar{\bT}
= \nu( \bT )$ is dual to~$\bar{\bT}^*$. Hence $T$ is cyclic of order 
$q^n - {\varepsilon}^n$. 
Let $\theta \in \Irr(T)$ such that the pairs $(\bT, \theta )$ and $(\bT^*, s )$ 
correspond via duality, and put
$\chi := \varepsilon_{\bT}\varepsilon_{\bG}R_{\bT}^{\bG}( \theta )$. 
Then $\chi'$ is an irreducible constituent of~$\chi$ by
Lemma~\ref{GeneralizedRegularBlocks}. 
Since $Z^{\circ}( \bar{\bG} )^{\bar{F}} = 
Z( \bar{\bG} )^{\bar{F}} = Z( \bar{G} )$ has order $q - \varepsilon$,
we conclude from Lemma \ref{RelativeRankAndIsogenies}(e) and our assumption
that $p \mid |Z( G )|$.
In particular,~$\bB$ is nilpotent, and so the hypotheses of
Lemma~\ref{StrictlyRegularBlocksCorCyclic} are satisfied. This then implies 
that~$\chi'$, being an element of $\mathcal{E}( G, s )$, is the non-exceptional 
character in~$\bB$.

Let~$\bar{\chi}$ denote the non-exceptional character of~$\bar{\bB}$.
By Corollary~\ref{StrictlyRegularBlocksCor} there is $\bar{\theta} \in
\Irr( \bar{T} )$ of $p'$-order, such that 
$\bar{\chi} = \varepsilon_{\bar{\bG}}\varepsilon_{\bar{\bT}}
R_{\bar{\bT}}^{\bar{\bG}}( \bar{\theta} )$.
Suppose that $|D| = p^l$ and let~$t$ be a generator of~$D$. Put 
$\bar{t} := \nu(t)$. As $Z( \bG^* )$ and $Z( \bar{\bG}^* )$ are connected,
the centralizers in~$\bG$, respectively in~$\bar{\bG}$, of the powers of~$t$, 
respectively of~$\bar{t}$, are connected. By \cite[Theorem~$1.3.10$(e)]{GeMa} 
and Lemma~\ref{RelativeRankAndIsogenies}(d) we get 
$\omega_{\bar{\bG}}^{[l]}( \bar{t} ) = \omega_{\bG}^{[l]}( t )$.
Then $\sigma_{\bar{\chi}}^{[l]}( \bar{t} ) = \sigma_{\chi}^{[l]}( t )
= \sigma_{\chi'}^{[l]}( t )$, where the first equality arises from
Lemma~\ref{SignAndOmegaInvariant}, and the second one from the last
statement of Lemma~\ref{GeneralizedRegularBlocks}. 
Using Remark~\ref{SignSequenceRemark}, we conclude that
$W( \bB ) \cong W( \bar{\bB} )$.
\end{proof}

\begin{cor}
\label{GLnCor}
Let $(\bG,F)$ be as in {\rm Case~$1$} of 
{\rm Table~\ref{ClassicalAlgebraicGroups}},
so that $G = \GL_n^{\varepsilon}(q)$. Suppose that $p \mid q - \varepsilon$,
and let~$\bB$ be a $p$-block of~$G$ with a cyclic defect group~$D$.

Then~$\bB$ is regular with respect to a cyclic maximal torus~$T$ of~$G$ of order
$q^n - \varepsilon^n$. In particular, we may assume that~$D$ is a Sylow
$p$-subgroup of~$T$. 

Let~$a$ be the positive integer such that $p^a$ is the highest power of~$p$
dividing $q - \varepsilon$, and define the non-negative integer~$a'$ by 
$n = m p^{a'}$ with $p \nmid m$. Then $|D| = p^{a+a'}$.
\end{cor}
\begin{proof}
The first part follows from Lemma~\ref{BlocksAndIsogeny}, if we let
$\bG = \bar{\bG}$ and~$\nu$ the identity map. The statement on~$|D|$ follows 
from this.
\end{proof}

The statements in Corollary~\ref{GLnCor} were first proved by Fong and 
Srinivasan in~\cite{fs2}. Notice however, that their notation differs from the 
one used here. Most notably, the group we denote by $\GU_n( q )$ here, is 
denoted by $U(n, \mathbb{F}_{q^2} )$ or $U(n,q^2)$ in~\cite{fs2}. 

To make use of the last part of Lemma~\ref{BlocksAndIsogeny}, we need to 
investigate when $\GL_n^{\varepsilon}(q)$ has a regular block with respect to
the cyclic maximal torus~$T$ of order~$q^n - \varepsilon^n$. We may realize
$G = \GL_n^{\varepsilon}(q)$ as $G = \bG^F$ with $\bG = \GL_n( \mathbb{F} )$ 
and some suitable Steinberg morphism~$F$ of~$\bG$. As~$\bG$ has connected 
center, the notions of regular and strictly regular blocks for~$G$ coincide; see 
the remarks following Definition~\ref{RegularBlocksDefinition}. We may also 
identify~$G$ with its dual group and~$T$ with its dual torus~$T^*$. Then a 
regular block of~$G$ with respect to~$T$ exists, if and only if~$T$ contains a 
regular $p'$-element; see Lemma~\ref{RegularElementsAndRegularEmbeddings}.
The last part of the following lemma is only used at a later stage of our work.

\begin{lem}
\label{ElementsInGeneralPosition}
Let $\bG$ be a connected reductive algebraic group over~$\mathbb{F}$ such that
$[\bG,\bG]$ is simply connected, and let~$F$ be a Steinberg morphism of~$\bG$ 
such that $G = \bG^F = \GL_n^{\varepsilon}( q )$ for some $n > 2$. (E.g.\ 
$\bG = \GL_n( \mathbb{F} )$.)
Assume that $p \mid q - \varepsilon$ and that $(q,n) \neq (2,3)$.

Then there is a prime~$f$ with $f \mid q^n - \varepsilon^n$, but
$f \nmid q^j - \varepsilon^j$ for all $1 \leq j < n$.

Let~$T$ be a cyclic maximal torus of~$G$ of order $q^n - \varepsilon^n$, and let 
$s \in T$ be of order~$f$. Then~$s$ is regular in~$\bG$ and strictly regular 
in~$G$; moreover, the image~$s'$ of~$s$ 
in~$\PGL_n^{\varepsilon}(q)$ is strictly regular; i.e.\ 
$C_{\PGL_n^{\varepsilon}(q)}( s' )$ is a maximal torus 
of~$\PGL_n^{\varepsilon}(q)$; see 
{\rm Definition~\ref{RegularBlocksDefinition}(a)}.
\end{lem}
\begin{proof}
Define the integer~$m$ by
$$m := 
\begin{cases}
n, & \text{if\ } \varepsilon = 1, \\
2n, & \text{if\ } \varepsilon = -1 \text{\ and\ } n \text{\ is odd}, \\
n, & \text{if\ } \varepsilon = -1 \text{\ and\ } 4 \mid n, \\
n/2, & \text{if\ } \varepsilon = -1 \text{\ and\ } n \equiv 2\,\,(\mbox{\rm mod}\,\,4). 
\end{cases}
$$
Then $m \geq 3$. If $q = 2$, then $\varepsilon = -1$, as~$p$ is a prime dividing
$q - \varepsilon$. Hence $(q,m) = (2,6)$ corresponds to the case
$(q,n) = (2,3)$, which is excluded. It follows that there exists a
primitive prime divisor~$f$ of $q^m - 1$, i.e.~$f$ is a prime with
$f \mid q^{m} - 1$ but $f \nmid q^j - 1$ for $1 \leq j < m$;
see \cite[Theorem~IX.$8.3$]{HuBl82}.
It is easy to check that $f \nmid q^j - \varepsilon^j$ for all $1 \leq j < n$.

Clearly~$s$ is regular in~$\bG$, as every maximal torus of~$G$ containing~$s$ is 
$G$-conjugate to~$T$. Since $[\bG,\bG]$ is simply connected by hypothesis, we
have $A_{\bG}(s) = \{ 1 \}$, so that~$s$ is strictly regular in~$G$. Taking
$\bG = \GL_n( \mathbb{F} )$, the statement about~$s'$ follows 
from~\ref{ComponentGroups}, as $|A_{\PGL_n( \mathbb{F} )}( s' )^F|$ divides 
$\gcd( q - \varepsilon, n )$. 
\end{proof}

We can now formulate the main result of this section.
\begin{thm}
\label{MainCorReduction}
Let~$\bG$ and~$F$ be as in one of the Cases $2$--$4$ of 
{\rm Table~\ref{ClassicalAlgebraicGroups}}. Let~$\bB$ be a $p$-block of~$G$ 
with a cyclic defect group~$D$. Assume that $W( \bB ) \not\cong k$. 

Then there is $\varepsilon \in \{ -1, 1\}$, an integer~$n' \geq 3$, a power~$q'$ 
of~$q$ with $p \mid q' - \varepsilon$, and a block~$\bB'$ 
of~$\GL^{\varepsilon}_{n'}(q')$, regular with respect to the cyclic maximal 
torus~$T'$ of~$\GL^{\varepsilon}_{n'}(q')$ of order 
${q'}^{\,n'} - \varepsilon^{n'}$, such that~$D$ is isomorphic to a Sylow 
$p$-subgroup~$D'$ of~$T'$, and $W( \bB ) \cong W( \bB' )$, if~$D$ and $D'$ are 
identified.
\end{thm}
\begin{proof}
Use the notation of the previous subsections. Let $t_1 \in D$ denote an element 
of order~$p$ and put $\bar{t}_1 = \nu(t_1)$. By Lemmas~\ref{NoEigenvalueIs1}
and~\ref{ShapeOfMinimalPolynomial}, we may assume that~$\bar{t}_1$ has no 
non-trivial fixed point on~$V$. Let us assume this from now on. 
Lemma~\ref{NonTypeAClassicalGroups} asserts that~$\nu$ restricts to an isogeny 
$C_{\bG}( t_1 ) \rightarrow C_{\bar{\bG}}( \bar{t}_1 )$ with kernel of 
order at most two. By Lemma~\ref{NonTypeAClassicalGroups0}, the group 
$C_{\bar{\bG}}( \bar{t}_1 )$ is isomorphic to a direct product of general linear 
groups of degree~$m$, and $C_{\bar{G}}( \bar{t}_1 ) \cong 
\GL_m^{\varepsilon}(q^d)$, for some $\varepsilon \in \{ - 1, 1 \}$, where
$m, d$ satisfy $2md = n$. Put $n' := m$ and $q' := q^d$. Then 
$p \mid q' - \varepsilon$ by Remark~\ref{RemNonTypeAClassicalGroups0}. Thus the
hypotheses of Lemma~\ref{BlocksAndIsogeny} hold for~$\nu$,~$n'$ and~$q'$.

In order to find~$W(\bB)$ we may replace~$G$ by $C_G(t_1)$ and~$\bB$ by a Brauer
correspondent~$\bc$ of~$\bB$ in $C_G( t_1 )$, but keeping~$D$. 
Lemma~\ref{BlocksAndIsogeny} shows that there is an $F$-stable maximal 
torus~$\bT$ of~$C_{\bG}( t_1 )$ such that~$D$ is a Sylow $p$-subgroup of~$T$,
and such that $T' = \nu(T)$ is cyclic of order ${q'}^{\,n'} - \varepsilon^{n'}$. 
Recall that $q' - \varepsilon$ divides $|Z( C_G( t_1 ) )|$ by 
Lemma~\ref{RelativeRankAndIsogenies}(e). Thus $n' \leq 2$ would imply 
$D \leq Z( C_G( t_1 ) )$, and then $W( \bc ) \cong k$ by 
\cite[Lemma~$3.6$(b)]{HL24}, a case we have excluded. Hence $n' > 2$.

Suppose first that $(q',n') = (2,3)$. In this case, $G = \Sp_6(2)$ and $p = 3$. 
But the cyclic $3$-blocks of~$\Sp_6(2)$ have defect~$1$ or~$0$; 
see~\cite{GAP04}. Then $W( \bB ) \cong k$ by \cite[Lemma~$3.6$(b)]{HL24}, 
contrary to our hypothesis. We may thus assume that $(q',n') \neq (2,3)$. By 
Lemma~\ref{ElementsInGeneralPosition}, there is an element $s \in T'$, which is
regular in~$C_{\bar{\bG}}( \bar{t}_1 )$ and strictly regular in 
$C_{\bar{G}}( \bar{t}_1 )$ with respect to~$T'$. In particular, there is a 
regular block~$\bB'$ of $C_{\bar{G}}( \bar{t}_1 )$ with respect to~$T'$, so 
that $W(\bB) \cong W( \bc ) \cong W( \bB' )$ by Lemma~\ref{BlocksAndIsogeny}. 
\end{proof}

\section{The general linear and unitary groups}
\label{TheGeneralLinearAndUnitaryGroups}
As a next step we consider the general linear and unitary groups. Fix a sign
$\varepsilon \in \{ 1, -1 \}$ and let~$\bG$ and $F = F_{\varepsilon}$ be as in 
Case~$1$ of Table~\ref{ClassicalAlgebraicGroups}. Then $G = \bG^F = 
\GL_n^{\varepsilon}(q)$. We may and will identify~${\bG}$ with its dual
group~${\bG}^*$. 

By Corollary~\ref{GLnCor}, if~$p$ is an odd prime with 
$p \mid q - \varepsilon$, a cyclic $p$-block of~${G}$ is regular with respect 
to the cyclic maximal torus~$T$ of~${G}$ of order $q^n - \varepsilon^n$. If 
$n = 1$, the principal block of~$G$ has this property. Suppose that $n = 2$ and 
$G \neq \GU_2(2)$. Then if $p \mid q - \varepsilon$, an element 
$\theta \in \Irr(T)$ of order $(q^2-1)_{p'}$ is in general position, so that 
regular blocks with respect to~$T$ also exist in this case. Moreover, if 
$n \geq 3$, such regular blocks exist, unless $G = \GU_3(2)$ and $p = 3$; see 
Lemma~\ref{ElementsInGeneralPosition}.

\subsection{The crucial case}
\label{TheCrucialCase}
Throughout this subsection we let~$p$ be an odd prime satisfying
$p \mid q - \varepsilon$, and we denote by~$a$ the positive integer such
that~$p^a$ is the highest power of~$p$ dividing~$q - \varepsilon$.
We will use the parameter~$\delta$ in the meaning of 
Subsection~\ref{TheRelevantClassicalGroups}. Thus, $\delta = 1$, if
$\varepsilon = 1$, and $\delta = 2$, if $\varepsilon = -1$.

The following lemma on certain $p$-elements of~$G$ will also be used in the 
sequels to this paper. The terms minimal polynomial and eigenvalue refer to the 
action of~$G$ on its natural vector space $V = \mathbb{F}^n_{q^{\delta}}$.

\begin{lem}
\label{MinimalPolynomialAndDeterminant}
Let $t \in G$ be semisimple and let~$\Gamma$ denote the minimal polynomial 
of~$t$ over $\mathbb{F}_{q^{\delta}}$. Let~$T$ denote a cyclic maximal torus 
of~$G$ of order $q^n - \varepsilon^n$.

{\rm (a)} If~$\Gamma$ is irreducible, the minimal polynomial of~$t^j$ is 
irreducible for all integers~$j$. (This does not use the hypothesis
that $p \mid q - \varepsilon$.)

{\rm (b)} Suppose that $t \in T$. Then~$\Gamma$ is irreducible, unless 
$\varepsilon = -1$, $n$~is even and $\Gamma = \Delta\Delta^{\dagger}$ for a 
monic irreducible polynomial $\Delta$ with $\Delta \neq \Delta^{\dagger}$.

{\rm (c)} Let~$\Delta$ be a monic irreducible polynomial 
over~$\mathbb{F}_{q^{\delta}}$, and let $\zeta \in \mathbb{F}$ be a root 
of~$\Delta$. If~$\zeta$ is of $p$-power order, then
$\deg(\Delta) = p^b$ with
$$
b = \begin{cases} 0, & \text{\ if\ } |\zeta| \leq p^a \\
                    a', & \text{\ if\ } |\zeta| = p^{a + a'} \text{\ with\ } a' \geq 0.
\end{cases}
$$

{\rm (d)} If $t \in T$ is of $p$-power order, then~$\Gamma$ is irreducible.

{\rm (e)} Suppose that~$\Gamma$ is irreducible and put $d := \deg( \Gamma )$.
Thus $n = n' d$ for some positive integer~$n'$. Let~$\zeta \in \mathbb{F}$ be 
an eigenvalue of~$t$. Then
$$\det(t) = \zeta^{n'(q^{\delta d} -1)/(q^{\delta}-1)}.$$
\end{lem}
\begin{proof}
(a) This is certainly well known. For convenience, we sketch a proof.
As~$\Gamma$ is irreducible, the subalgebra $\mathbb{F}_{q^{\delta}}[t]$ of the 
full matrix algebra is a field. It is, in particular, a finite extension 
of~$\mathbb{F}_{q^{\delta}}$. Thus~$t^j$ is algebraic 
over~$\mathbb{F}_{q^{\delta}}$, and so $\mathbb{F}_{q^{\delta}}[t^j]$ is a 
field. This implies the claim.

(b) Suppose that $\varepsilon = 1$ or that~$n$ is odd. Then~$T$ lies in a 
Coxeter torus of~$\GL_n( q^{\delta} )$. Hence~$t$ is a power of an element 
of~$\GL_n( q^{\delta} )$ acting irreducibly on~$V$. 
Thus~$\Gamma$ is irreducible by~(a). 

Now suppose that $\varepsilon = -1$ and that~$n = 2m$ is even. Then~$T$ is a
Coxeter torus of a Levi subgroup $L \cong \GL_m( q^2 )$ of~$G$. Now 
$V = V_1 \oplus V_2$ with $L$-invariant totally isotropic 
subspaces~$V_1$, $V_2$ of~$V$, such that~$L$ acts on~$V_1$ in 
the natural way, i.e.\ as $\GL_m( q^2 )$ acts on~$\mathbb{F}^m_{q^2}$, whereas 
the action of~$L$ on~$V_2$ is the natural action twisted by an automorphism 
of~$L$. Thus, by the first part of the proof, if~$\Gamma$ is reducible, we have 
$\Gamma = \Delta \Delta'$ with monic irreducible polynomials 
$\Delta \neq \Delta'$. As~$t$ lies in the unitary group, we must have 
$\Delta' = \Delta^{\dagger}$.

(c) Put $d = \deg( \Delta )$. If $|\zeta| \leq p^a$, then 
$\zeta \in \mathbb{F}_{q^{\delta}}$, and thus $d = 1$. 
Suppose then that $|\zeta|= p^{a + a'}$ for some non-negative integer~$a'$. 
Observe that the highest power of~$p$ dividing $q^{\delta} - 1$, respectively 
$q^{\delta p^{a'}} - 1$, equals~$p^a$, respectively $p^{a+a'}$. 
Hence
$d = [\mathbb{F}_{q^{\delta}}[\zeta]\colon\!\mathbb{F}_{q^{\delta}}] = p^{a'}$. 

(d) Suppose that~$\Gamma$ is reducible, and let $\zeta \in \mathbb{F}$ denote
a root of~$\Delta$, where~$\Delta$ is as in~(b). Then~$\zeta$ is of $p$-power 
order, and thus $\deg(\Delta)$ is odd by~(c). But then 
$\Delta = \Delta^{\dagger}$ by Lemma~\ref{UnitaryDelta}, a contradiction.

(e) As $\Gamma$ is irreducible,
the roots of~$\Gamma$ are
$$\zeta, \zeta^{q^{\delta}}, \zeta^{q^{\delta \cdot 2}}, \ldots ,
\zeta^{q^{\delta \cdot(d - 1)}},$$
and the product of these roots equals
$$\zeta^{(q^{\delta d} - 1)/(q^{\delta} - 1)}.$$
This proves our claim.
\end{proof}

We will also need the following elementary result.

\begin{lem}
\label{CongruenceOfFloors}
Let $m, b$ be positive integers. Then the following hold.

{\rm (a)} If $m$ is even or if $p \equiv 1\,\,(\mbox{\rm mod}\,\,4)$,
then $\lfloor m/2 \rfloor + \lfloor mp^b/2 \rfloor$ is even.

{\rm (b)} If $m$ is odd and $p \equiv -1\,\,(\mbox{\rm mod}\,\,4)$,
then $\lfloor m/2 \rfloor + \lfloor mp^b/2 \rfloor 
\equiv b\,\,(\mbox{\rm mod}\,\,2)$.
\end{lem}
\begin{proof}
Suppose first that~$m$ is even. Then
$\lfloor m/2 \rfloor + \lfloor mp^b/2 \rfloor = m(1+p^b)/2$ is
even as~$p$ is odd. This gives the first part of~(a).

Suppose now that~$m$ is odd. Then
$\lfloor m/2 \rfloor + \lfloor mp^b/2 \rfloor = m(1+p^b)/2 - 1$. This
is even if and only if $(1+p^b)/2$ is odd, i.e.\ if and only if
$p^b\equiv 1\,\,(\mbox{\rm mod}\,\,4)$. This proves the remaining results.
\end{proof}

Assume from now on that $n \geq 2$.
Define the integers~$a'$ and~$m$ by $n = p^{a'}m$ and $p \nmid m$.
Then~$a'$ is non-negative and~$m$ is positive.
Let~${\bB}$ denote a $p$-block of~${G}$ with a non-trivial cyclic 
defect group~${D}$. Let ${\bT}$ denote an $F$-stable maximal torus 
of~${\bG}$ such that~${T}$ is cyclic of order $q^n - \varepsilon^n$. 
(Notice that~${\bT}$ is a Coxeter torus of~${\bG}$, unless 
${G} = \GU_n(q)$ and~$n$ is even; in any case,~${\bT}$ is uniquely 
determined in~${\bG}$ up to conjugation in~${G}$.) 
By Corollary~\ref{GLnCor} we may assume that~${D}$ is a Sylow 
$p$-subgroup of~${T}$, which implies $|{D}| = p^{a+a'}$. Also, there is 
$\theta \in \Irr( {T} )$ in general position and of $p'$-order, such that 
the unique non-exceptional character of~$\Irr({\bB})$ is the character 
${\chi} :=
\varepsilon_{{\bG}}\varepsilon_{{\bT}}R_{{\bT}}^{{\bG}}( \theta )$;
see Corollary~\ref{StrictlyRegularBlocksCor}.

\begin{lem}
\label{SignSequenceLem}
Assume the notation and assumptions introduced in the previous paragraph.
Let $t \in {T}$ be a non-trivial $p$-element and let~$l$ be a positive 
integer such that $t^{p^{l-1}} \in Z( {G} )$. Put 
$\Lambda := \{ 0, 1, \ldots, l - 1 \}$. 

Then $\sigma_{{\chi}}^{[l]}( t ) = \omega_{\Lambda}( \mathbf{1}_I )$, for 
an interval $I \subseteq \Lambda \setminus\{0\}$. 

Moreover, $I = \emptyset$, except if $\varepsilon = -1$, $n$ is odd, 
$p \equiv -1\,\,(\mbox{\rm mod}\,\,4)$ and $|t| = p^{a+a''}$ for some 
$1 \leq a'' \leq a'$. In this case, $I = [l - a'', l - 1]$.
\end{lem}
\begin{proof}
By Lemma~\ref{SignAndOmegaInvariant} we have $\sigma_{{\chi}}^{[l]}( t ) = 
\omega_{{\bG}}^{[l]}( t )$. Choose $j \in \Lambda \setminus \{0\}$ and
let $u := t^{p^{l - j}}$. We have to compute $\omega_{{\bG}}( u )$.

Clearly, $C_{{G}}( u ) = {G}$ if $|u| \mid p^a$, in which case 
$\omega_{{\bG}}( u ) = 1$. Suppose that $|u| = p^{a+b}$ for some integer 
$b > 0$.  As $|{D}| = p^{a+a'}$, we have $b \leq a'$. By 
Lemma~\ref{MinimalPolynomialAndDeterminant}, the minimal polynomial of~$u$ is
irreducible, and the eigenvalues of~$u$ span a field extension
of~$\mathbb{F}_{q^{\delta}}$ of degree~$p^b$.
Thus $C_{{G}}( u ) \cong \GL^{\varepsilon}_{mp^{a'-b}}( q^{p^b} )$. 
By Example~\ref{OmegaOfGLNQ}, we get 
$$\omega_{{\bG}}( u ) = (-1)^{mp^{a'} + mp^{a'-b}}$$
if $\varepsilon = 1$,
and 
$$\omega_{{\bG}}( u ) = (-1)^{\lfloor mp^{a'}/2 \rfloor + \lfloor mp^{a'-b}/2 \rfloor}$$
if $\varepsilon = -1$.

Thus $\omega_{{\bG}}( u ) = 1$ if $\varepsilon = 1$. Suppose that
$\varepsilon = - 1$. Lemma~\ref{CongruenceOfFloors} then implies that 
$\omega_{{\bG}}( u ) = 1$ unless $m$ and $b$ are odd 
and $p \equiv -1\,\,(\mbox{\rm mod}\,\,4)$, in which case 
$\omega_{{\bG}}( u ) = -1$.

Suppose now that $\varepsilon = -1$, that $m$ is odd, that
$p \equiv -1\,\,(\mbox{\rm mod}\,\,4)$ and that $|t| = p^{a + a''}$ for some
positive integer~$a''$. As $t^{p^{l-1}} \in Z({G})$ by hypothesis, we have 
$a'' < l$. By what we have proved so far, we get 
$$\sigma_{{\chi}}^{[l]}( t ) = (1, \ldots, 1, -1, 1, -1, \ldots, \pm 1),$$
where the first $-1$ appears at position $l - a'' + 1$.
The claim now follows from Lemma~\ref{OmegaOfInterval}.
\end{proof}

We are now able to determine $W( {\bB} )$. By Lemma~\ref{lem:OpZG}(c), this is 
of the form $W_D( A )$ for some $A \subseteq \{ a, \ldots, a + a' - 1 \}$.

\begin{cor}
\label{GLNQdeq1Source}
Assume the notation and assumptions introduced in the paragraph preceding
{\rm Lemma~\ref{SignSequenceLem}}.

Then $W({\bB}) \cong k$, unless $\varepsilon = -1$,~$n$ is odd and 
$p \equiv -1\,\,(\mbox{\rm mod}\,\,4)$. In this case,
$W({\bB}) \cong W_{{D}}([a,a+a' - 1])$.
\end{cor}
\begin{proof}
Let~$t$ denote a generator of~${D}$. Then $|t| = p^{a + a'}$. 
Put $l := a + a'$ and let $\Lambda := \{ 0, \ldots, l - 1 \}$.
Lemma~\ref{SignSequenceLem} yields
$\sigma_{{\chi}}^{[l]}( t ) = \omega_{\Lambda}( \mathbf{1}_I )$ with 
$I = \emptyset$ and $I = [a, a+a' - 1]$ in the respective cases.

Our assertion follows from Lemma~\ref{NewNotation} and \cite[Lemma~$3.3$]{HL24}.
\end{proof}

\subsection{The general case}
Here, we will show how to reduce the computation of~$W({\bB})$ for an 
arbitrary cyclic block~${\bB}$ of~${G}$ to 
Corollary~\ref{GLNQdeq1Source}. We will also show, 
by way of an example, how to use Corollary~\ref{GLNQdeq1Source} to construct a 
non-uniserial cyclic block~${\bB}$ of a suitable group $\GU_n(q)$ with 
$W(\bB) \not\cong k$. We will also discuss an important consequence of 
Corollary~\ref{GLNQdeq1Source} to the special linear and unitary groups.

Keep the notation introduced at the beginning of 
Section~\ref{TheGeneralLinearAndUnitaryGroups}. In addition, fix an odd 
prime~$p$ with $p \nmid q$. To allow for a uniform treatment, we denote by~$d$ 
the order of $\varepsilon q$ modulo~$p$. 
Then $p \mid q^d + 1$ if $\varepsilon = -1$ and~$d$ is odd, and we say that~$p$ 
is a \textit{unitary prime} for $\GU_n( q )$. If $\varepsilon = -1$ and $d$ is 
even, then $p \mid q^d - 1$, and~$p$ is called a \textit{linear prime} for 
$\GU_n( q )$.

\begin{thm}
\label{GLNQdneq1}
Let~$d$ denote the order of~$\varepsilon q$ modulo~$p$. Assume that $d > 1$.
Let~${\bB}$ denote a $p$-block of~${G}$ with a non-trivial cyclic 
defect group~${D}$. Assume that the fixed space of~${D}$ on the 
natural vector space for~${G}$ is trivial.
Put ${L} := C_{{G}}( {D}_1 )$, where~${D}_1$ denotes the 
unique subgroup of~${D}$ of order~$p$.

Then there are non-negative integers $m, a'$ with $p \nmid m$ and 
$n = mdp^{a'}$, such that 
${L} \cong \GL_{mp^{a'}}^{\varepsilon(d)}( q^{d} )$, where
$\varepsilon(d) = 1$ if $\varepsilon = 1$ or if $\varepsilon = -1$ and~$d$ is 
even, and $\varepsilon(d) = -1$, otherwise.

Let ${\bc}$ denote a Brauer correspondent of~${\bB}$ in~${L}$.  
Then $W( {\bB} ) \cong W( {\bc} )$ and $W( {\bc} )$ can be
computed from 
{\rm Corollary~\ref{GLNQdeq1Source}}.
\end{thm}
\begin{proof}
Let $t \in {D}$ denote a generator of~${D}$ and let~$t_1$ be a power of~$t$ 
generating~${D}_1$. By assumption,~$t$ has no non-trivial fixed vector on the 
natural vector space for~${G}$. Let~$\Gamma$ and $\Gamma_1$ denote the minimal
polynomials of~$t$, respectively~$t_1$, over $\mathbb{F}_{q^{\delta}}$. Then 
\cite[Theorem (3C)]{fs2} implies that~$\Gamma$ has at most two irreducible 
factors. (Although \cite{fs2} assumes that~$q$ is odd, this result also holds 
for even~$q$, e.g.\ by \cite[Theorem~$6.1.2$]{FLO} and the fact that~$D$ is a 
cyclic radical $p$-subgroup of $C_G(D)$.)

The roots of~$\Gamma_1$, respectively $\Gamma$, span field extensions of 
$\mathbb{F}_{q^{\delta}}$ of degree~$d$, respectively $dp^{a'}$ for some 
non-negative integer~$a'$. Thus \cite[Theorem~$6.1.2$]{FLO} implies that
$C_{{G}}( {D} ) \cong \GL_m^{\varepsilon(d)}( q^{dp^{a'}} )$ for a positive
integer~$m$ such that $n = mdp^{a'}$. The fact that~$D$ is a defect group of
$C_G( D )$ yields $p \nmid m$, by the arguments given in the second part
of the proof of \cite[Proposition (4B)]{fs2}, which also apply for even~$q$. The 
structure of~${L}$ follows from this. Indeed,~$t_1$ is a power of~$t$, and thus
$\Gamma$ and $\Gamma_1$ have the same number of irreducible factors by 
Lemma~\ref{MinimalPolynomialAndDeterminant}(a). The assertion 
about~$W( {\bc} )$ is clear, since $p \mid q^d - \varepsilon(d)$.
\end{proof}

As a further application of Corollary~\ref{GLNQdeq1Source}, we provide an
explicit example demonstrating that the Bonnaf{\'e}--Dat--Rouquier Morita 
equivalence~\cite{BoDaRo} is not a source algebra equivalence in general.
It also gives an example for a non-uniserial block~${\bB}$ with 
$W( {\bB} ) \not\cong k$.
\begin{exmp}
\label{BoDaRoExample}
{\rm
Assume that $\varepsilon = -1$ and that $p \equiv -1\,\,(\mbox{\rm mod}\,\,4)$.
Let~$d$ be an odd, positive integer with $d > 1$. Suppose that the order of~$-q$ 
modulo~$p$ equals~$d$ (given~$p$ and~$d$, this is a condition on~$q$). Assume 
also that $n = pd$.

Let~${\bT}$ denote a Coxeter torus of~${\bG}$ and let~${D}$ be a Sylow 
$p$-subgroup of~${T}$. Then $C_{{\bG}}( {D} ) = {\bT}$, as a generator of~$D$ 
has $dp = n$ distinct eigenvalues. Write~${D}_1$ for the subgroup of~${D}$ of 
order~$p$ and put ${\bL} := C_{{\bG}}( {D}_1 )$. Then ${L} \cong \GU_p( q^d )$.

Now suppose that $s \in {T}$ is a $p'$-element with $C_{{G}}( s ) 
\cong \GU_d( q^p )$ (the eigenvalues of~$s$ span an extension field of 
$\mathbb{F}_{q^2}$ of degree~$p$). Identify~$\bT$ with its dual torus 
$\bT^* \leq \bG^* = \bG$. Let $\theta \in \Irr( T )$ such that the pairs
$(T,s)$ and $(T, \theta)$ are in duality. Then 
$[N_{{G}}( {\bT}, \theta )\colon\!{T}] = d$, and thus is prime to~$p$.
(Recall that $N_{{G}}( {\bT}, \theta )$ denotes the stabilizer of~$\theta$
in~$N_G(\bT)$.) As $C_{\bG}( D ) = \bT$, we have $N_G( \bT )= N_G( T )$.
Let ${\bB}$ denote the $p$-block of~${G}$ with defect group~${D}$ 
corresponding to the $p$-block of~${T}$ determined by~$\theta$. Now 
$s \in {T}$ is regular in~${\bL}$, and thus there is a block~${\bc}$
of~${L}$ such that $\Irr({\bc})$ contains the irreducible 
Deligne--Lusztig character
$\varepsilon_{{\bT}}\varepsilon_{{\bL}} R_{{\bT}}^{{\bL}}( \theta )$;
see Lemma~\ref{GeneralizedRegularBlocks}.
Then~${\bc}$ is a Brauer correspondent of~${\bB}$ and thus
$W( {\bB} ) \cong W( {\bc} )$.
By Corollary~\ref{GLNQdeq1Source}, where $({G}, n, q, {\bB})$ now
take the values $({L}, p, q^d, {\bc})$, we have
$W( {\bc} ) \not\cong k$.

By the results~\cite{fongcl} of Fong and Srinivasan, the Brauer tree
of~${\bB}$ has exactly~$d$ edges and is a straight line. In
particular,~${\bB}$ is not uniserial.

The Bonnaf{\'e}--Dat--Rouquier reduction \cite{BoDaRo} establishes a Morita
equivalence between~${\bB}$ and a unipotent block~${\bb}$ of 
$C_{{G}}( s ) \cong \GU_d( q^p )$. By~\cite[Theorem (7A)]{fs2}, the 
block~${\bb}$ is the principal block of~$C_{{G}}( s )$, as~$d$ is 
also equal to the order of $-q^p = (-q)^p$ modulo~$p$, and thus the
unipotent characters of $C_{{G}}( s )$ either lie in the principal block or 
are of defect~$0$. Hence $W( {\bb} ) \cong k$. 
Thus the Bonnaf{\'e}--Dat--Rouquier reduction does not preserve the source 
algebra equivalence class of blocks in general.

A specific instance for the parameters $(d,p,q,|s|)$ as above is given by
$(3,7,5,449)$. In this particular case, $W({\bc}) = W_D( [1,1] )
= \Omega_{D/D_1}( k )$, where~$D$ has order~$7^2$.
}\hfill $\Box$
\end{exmp}

Let us discuss an important consequence of the above investigations to the 
special linear and unitary groups. 

\begin{rem}
\label{SLNQdneq1}
{\rm
Let $G' = \SL_n^{\varepsilon}( q ) \leq G = \GL_n^{\varepsilon}(q)$. Assume that 
$p \nmid q - \varepsilon$. 

(a) Let~$\bB'$ denote a $p$-block of~$G'$ with a 
non-trivial cyclic defect group, and let~$\bB$ be a $p$-block of~$G$ 
covering~$\bB'$. As $p \nmid q - \varepsilon$, any defect group of~$\bB'$ is 
also a defect group of~$\bB$ and we have $W( \bB' ) \cong W( {\bB} )$ by 
\cite[Lemma~$4.3$]{HL24}. By Lemma~\ref{NoEigenvalueIs1} and 
Theorem~\ref{GLNQdneq1}, the invariant~$W( {\bB} )$ can be computed with
Corollary~\ref{GLNQdeq1Source}. 

(b) Let~$\bar{\bB}'$ be a cyclic $p$-block of a central quotient~$\bar{G}'$ 
of~$G'$. By \cite[Lemma~$4.1$]{HL24}, there is a block~$\bB'$ of~$G'$ such that
$W( \bB' ) \cong W( \bar{\bB}' )$, and~$W( \bB' )$ can be computed by~(a).
}\hfill $\Box$
\end{rem}

Thus the computation of~$W( \bar{\bB}' )$ for a cyclic 
$p$-block~$\bar{\bB}'$ of a central quotient of~$\SL_n^{\varepsilon}( q )$
is reduced to the case $p \mid q - \varepsilon$. This will be settled in 
Part~III.

\section*{Acknowledgements}

The authors thank Olivier Dudas, Meinolf Geck, Radha Kessar, 
Burkhard K{\"u}ls\-ham\-mer,
Markus Linckelmann, Frank L{\"u}beck,
Klaus Lux, Gunter Malle and Jay Taylor for innumerable invaluable discussions
and elaborate explanations on various aspects of this work.
We are in particular indebted to Gunter Malle for carefully reading a first 
version of this manuscript, thereby detecting two lapses.
Finally, we thank the referee for her or his thorough evaluation and valuable 
suggestions, leading to a considerable improvement of the exposition.

\end{document}